\documentclass[a4paper,10pt]{article}
\usepackage[utf8]{inputenc}     % Si l'encodage du fichier est unicode UTF-8
\usepackage[T1]{fontenc}
\usepackage{lmodern}
\usepackage{graphicx}
\usepackage[normalem]{ulem}
\usepackage[english]{babel}
\usepackage{verbatim}
\usepackage{graphicx}
\usepackage{subfigure}
\usepackage{amsmath, amsfonts, dsfont, amssymb,amsthm} % Pour écrire des mathématiques	
\usepackage[noend]{algpseudocode}
\usepackage{hyperref}
\usepackage{url}
\usepackage{enumerate}
\usepackage{enumitem}
\usepackage[margin=3cm]{geometry}
\usepackage{color}
\usepackage{appendix}
\usepackage{fancyhdr}
\newtheorem{theorem}{Theorem}
\newtheorem{proposition}[theorem]{Proposition}
\newtheorem{definition}[theorem]{Definition}
\newtheorem{lemma}[theorem]{Lemma}
\newtheorem{remark}[theorem]{Remark}
\newtheorem{corollary}[theorem]{Corollary}

\newtheorem{fact}[theorem]{Fact}
\numberwithin{theorem}{section}

\newcommand{\btau}{\boldsymbol{\tau}}
\newcommand{\bsigma}{\boldsymbol{\sigma}}

\newcommand{\calD}{\mathcal{D}}
\newcommand{\calE}{\mathcal{E}}
\newcommand{\calF}{\mathcal{F}}

\newcommand{\calM}{\mathcal{M}}

\newcommand{\calO}{\mathcal{O}}

\newcommand{\calS}{\mathcal{S}}
\newcommand{\calT}{\mathcal{T}}
\newcommand{\calU}{\mathcal{U}}

\newcommand{\calW}{\mathcal{W}}

\newcommand{\bbE}{\mathbb{E}}

\newcommand{\bbN}{\mathbb{N}}

\newcommand{\bbP}{\mathbb{P}}

\newcommand{\ind}{\mathds 1}

\newcommand{\tv}{\mathsf{v}}
\newcommand{\tu}{\mathsf{u}}
\newcommand{\tw}{\mathsf{w}}

\newenvironment{tight_itemize}{
\begin{itemize}
  \setlength{\itemsep}{2pt}
  \setlength{\topsep}{0pt}
  \setlength{\parskip}{1pt}
  \setlength{\partopsep}{0pt}  
}{\end{itemize}}

\title{Influence of the seed in \textit{affine} preferential attachment trees}
\author{David Corlin Marchand\thanks{Laboratoire de Mathématiques d’Orsay, Univ. Paris-Sud, CNRS, Université Paris-Saclay, 91405 Orsay, France
and  
Universit\'e de Fribourg, 
23 Chemin du Mus\'ee, 
CH-1700 Fribourg, 
Switzerland.
\url{david.marchand@u-psud.fr}}
~and Ioan Manolescu\thanks{D\'epartement de Math\'ematiques, 
Universit\'e de Fribourg, 
23 Chemin du Mus\'ee, 
CH-1700 Fribourg, 
Switzerland.
\url{ioan.manolescu@unifr.ch}}}

\date\today

\begin{document}

\maketitle

\begin{abstract}
We study randomly growing trees governed by the affine preferential attachment rule. 
Starting with a seed tree $S$, vertices are attached one by one, each linked by an edge to a random vertex of the current tree, 
chosen with a probability proportional to an affine function of its degree. 
This yields a one-parameter family of preferential attachment trees $(T_n^S)_{n\geq|S|}$, of which the linear model is a particular case. 
Depending on the choice of the parameter, the power-laws governing the degrees in $T_n^S$ have different exponents. 

We study the problem of the asymptotic influence of the seed $S$ on the law of $T_n^S$.
We show that, for any two distinct seeds $S$ and $S'$, the laws of $T_n^S$ and $T_n^{S'}$ remain at uniformly positive total-variation distance as $n$ increases. 

This is a continuation of \cite{curien2015scaling}, which in turn was inspired by a conjecture of \cite{bubeck2015influence}.
The technique developed here is more robust than previous ones and is likely to help in the study of more general attachment mechanisms. 
\end{abstract}

\tableofcontents

\noindent

\section{Introduction}\label{introductionpaper}

\textit{Linear preferential attachment trees} - or \textit{Bar\'abasi-Albert trees} - are trees that grow randomly according to the following simple mechanism. 
%stem from a random growing and recursive model for finite trees.  
Start with a finite tree $S$ with $k$ vertices - we call it the \textit{seed tree} and $k$ its size - and construct inductively a sequence $(T_{n}^{S})_{n \geq k}$ of random trees, where $T_{n+1}^{S}$ is obtained from $T_n^S$ by adding a new vertex and connecting it to a vertex of $T_{n}^{S}$ chosen at random with a probability proportional to its degree.
Thus the tree $T_{n+1}^S$ has one more vertex than $T_n^S$ - namely $n+1$ - and one more edge - namely~$n$.  

The reader might note that this model formalises the adage "the rich get richer" since vertices with high degrees are more likely to receive new connections. This property suits a large class of real networks (see \cite{barabasi1999emergence}) mainly because of the emergence of power laws for the sequence of degrees of the tree as its size grows, see \cite{bollobas2001degree}. 
Such laws are observed in a wide range of contexts, like social networks (webgraph, citation graph, etc \cite{Wang08}, \cite{newman2018networks}) or even in biological networks such as interaction protein networks (see \cite{middendorf2005inferring}). The degree distribution in linear preferential attachment trees has been deeply investigated;
%notably the question of the convergence - once properly rescaled - of the maximum degree. First, it can be quite easily shown that the degree of any particular vertex grows as $n^{1/2}$. M\'ori has then proved in \cite{mori2005maximum} that the same result holds for the maximum degree and that furthermore the law of its limit is fully known. 
for an extensive overview of these traditional topics in the preferential attachment model's analysis, the reader is directed to \cite[Ch.~8]{hofstad_2016}. 

In the literature, there are numerous variations in the definition of preferential attachment model. 
Here we will consider one where the new vertex is attached to an old vertex chosen with probability which is an affine function of its degree, rather than a linear one. 
%Here we consider a modification of the Bar\'abasi-Albert tree model. 
Fix a parameter $\alpha > 0$ and $S$ a seed tree of size $k \geq 2$. 
The $\alpha$-PA trees grown from the seed $S$ are denoted  by $(T_{n}^{S})_{n \geq k}$ and are defined recursively as follows. 
Given $(T_{n}^{S})_{k \leq n \leq N}$ for some $N \geq k$ with $T_N^S = T$ for some fixed tree $T$, 
select randomly a vertex $u_{N}$ in $T$ with probability:
\begin{align}
	\label{probabilitylawvertexaffinePAmodel}
	\mathbb{P}(u_N = u \,\vert\,T_{k}^{S}, \dots, T_{N}^{S} \text{ with } T_{N}^{S}=T) = \frac{\deg_{T}(u)-1+\alpha}{(1+\alpha)N-2} \qquad \forall u\in T,
\end{align}
and link it to a new vertex by an edge. The resulting tree is $T_{N+1}^{S}$. 

Note that $(1+\alpha)N-2 = \sum_{u \in T} \deg_{T}(u)-1+\alpha$ for all trees $T$ of size $N$, so~\eqref{probabilitylawvertexaffinePAmodel} defines indeed a probability measure.
This model, which was first introduced in \cite{mori2002random}, is a generalisation of the Bar\'abasi-Albert model, in that the latter is obtained when choosing $\alpha = 1$. The authors of \cite{berger2005spread} remarked that transitions probabilities \eqref{probabilitylawvertexaffinePAmodel} provide a P\'{o}lya urn representation of $\alpha$-PA trees; particularly convenient properties follow, such as estimates on the degree growth and a form of exchangeability for the sequence $(u_N)_{N\geq k}$.
When the dependence on $S$ is not important, we will drop $S$ from the notation. 

\begin{center}
{\bf Henceforth $\alpha > 0$ is fixed.}
\end{center}

The same questions as for the linear model may be asked of the affine one. For instance, one may study the growth of degrees of given vertices in $T_n$. It may be proved (see \cite[Ch. 8.3 \& 8.4]{hofstad_2016} or Remark~\ref{rem:xi(u)} of this paper)  that the degree of any vertex of the seed increases at polynomial speed, with an exponent that depends on $\alpha$ and varies between $1$ and $0$ as $\alpha$ ranges from $0$ to infinity.
It may appear surprising that an apparently insignificant difference in the attachment mechanism leads to different scaling exponents in the power laws governing the degree sequence, and hence to preferential attachment models of different universality classes.
More substantial variations of the model will be discussed in Section~\ref{sec:open_problems}.

The problem of interest in this paper is that of the recognition of the seed.
Precisely, we will study whether the seed tree has any influence on the law of the tree obtained after a large number of iterations of the growth procedure. This question was asked (and partly solved) for the linear model in \cite{bubeck2015influence}. The complete answer for the linear model was obtained in \cite{curien2015scaling}, and for the uniform preferential attachment model in \cite{bubeck2017}. 
In both cases, the seed is shown to influence the asymptotic law of the model. Our aim is to generalise the result to the $\alpha$-PA, 
of which the uniform model may be perceived as a limit.

\begin{theorem}\label{thm:main}
    Let $S$ and $S'$ be two finite trees of size $k_1,k_2 \geq 3$. Then for any $\alpha >0$, the following limit:
    \begin{align}
    \label{limittotalvariationmaintheoem}
    d(S,S') = \lim\limits_{n \to +\infty} d_{TV}(T_{n}^{S},T_{n}^{S'})
    \end{align}
    exists and is non zero when $S$ is not graph-isomorphic to $S'$. 
\end{theorem}

As a consequence of Theorem~\ref{thm:main}, the function $d$ is a metric on finite trees with at least 3 vertices. 
It measures the statistical chance of distinguishing between two different seed trees given observations of the $\alpha$-PA trees grown from them.

\begin{remark}\label{remarkmaintheoremexistenceofthelimit}
    The existence of the limit in Theorem~\ref{thm:main} stems from the fact that the sequence 
    $d_{TV}(T_{n}^{S},T_{n}^{S'})$ is decreasing in $n$, which may be proved by a coupling argument. 
\end{remark}

In proving Theorem~\ref{thm:main}, we will look at the total variation between $T_{n}^{S}$ and $T_{n}^{S'}$ through some integer-valued observable obtained from them. The same is done in all previous studies \cite{bubeck2015influence,curien2015scaling,bubeck2017}. 

In \cite{bubeck2015influence}, where the model is linear (that is $\alpha = 1$) and $S$ and $S'$ are assumed to have different degree sequences, the authors use as observable the maximum degree. 
Indeed, they show that the maximal degrees of $T_n^{S}$ and $T_n^{S'}$ have different tail distributions, which then easily leads to~$d(S,S')>0$.
The key to this argument is that the degrees of vertices in the Bar\'abasi-Albert model evolve as a P\'olya urn. 
When the degree distributions of $S$ and $S'$ are identical, but their geometry is different, the observable of \cite{bubeck2015influence} is unable to distinguish between $T_n^{S}$ and $T_n^{S'}$. To overcome this difficulty, a more complex observable was introduced in \cite{curien2015scaling}: the number of embeddings of a fixed finite tree $\tau$ inside $T_n^{S}$, weighted by some function of the degrees in $T_n^S$ of the embedding of $\tau$. 
The same type of observable was used in \cite{bubeck2017} and will be used below. 

For $n \geq 0$ and $d \geq 1$, set $[n]_{d} = n(n-1)\dotsc(n-d+1)$. Also set $[n]_{0} = 1$ for any $n\geq 0$.
A {\em decorated tree} is a couple $\btau = (\tau,\ell)$ where $\tau$ is a tree and $\ell$ is function from the vertices of $\tau$ to the set of non-negative integers. For a decorated tree $\btau$ and a (bigger) tree $T$, set 
\begin{align}\label{eq:observable}
	\mathcal{F}_{\btau}(T):= \sum_{\phi: \tau \hookrightarrow T} \,\prod_{u \in \tau} [\deg_{T}(\phi(u))-1]_{\ell(u)},
\end{align}
where the sum is over all graph embeddings of $\tau$ in $T$ and the product is over all vertices of $\tau$. 

For two distinct seeds $S$ and $S'$, 
we aim to show that, for a well chosen decorated tree $\btau$,
the difference of the expectations of $\mathcal{F}_{\btau}(T_n^{S})$ and $\mathcal{F}_{\btau}(T_n^{S'})$ 
is of the same order as each of them and as their standard deviation. 
This allows to control the total variation between $T_n^{S}$ and $T_n^{S'}$ using the following bound. 

\begin{lemma}\label{lem:d_tv_moments}
	For any two real-valued, square-integrable random variables $X$ and $Y$,
	\begin{align*}%\label{lowerboundtotalvariationfirstsecondmoment}
    d_{TV}(X,Y) \geq 
    \frac{(\mathbb{E}[X]-\mathbb{E}[Y])^2}{(\mathbb{E}[X]-\mathbb{E}[Y])^2 + 2 \cdot \Big( \mathbb{E}[ X^2] + \mathbb{E}[ Y^2]\Big)}.
    \end{align*}
\end{lemma}

\begin{proof}[Proof of Lemma~\ref{lem:d_tv_moments}]
Let $(X',Y')$ be a coupling of the random variables $X$ and $Y$. By using Paley-Zigmund's inequality, then Jensen's one, we get:
\begin{align*}
\bbP(X' \neq Y') \geq \frac{(\bbE[X']-\bbE[Y'])^2}{\bbE[(X'-Y')^2]}.
\end{align*}
Furthermore, a simple decomposition gives:
\begin{align*}
\bbE[(X'-Y')^2] =& \ \bbE[\big(X'-\bbE[X']+\bbE[X']-\bbE[Y']+\bbE[Y']-Y'\big)^2] \\ \leq& \ 2 \cdot \big(\textrm{Var}(X') + \textrm{Var}(Y')\big) + \big(\bbE[X']-\bbE[Y']\big)^2 \\ \leq& \ 2 \cdot \big(\bbE[(X')^2] + \bbE[(Y')^2] \big) + \big(\bbE[X']-\bbE[Y']\big)^2.
\end{align*}
It finally remains to note that $X'$, resp. $Y'$, has the same moments as $X$, resp. $Y$, and that the total variation between $X$ and $Y$ is obtained by taking the infimum of $\bbP(X' \neq Y')$ over all couplings $(X',Y')$ of $X$ and $Y$.
\end{proof}

It comes to no surprise that the evolution of the moments of $\calF_{\btau}(T_n^S)$ is of great importance for the proof. 
Let us fix $\btau$ and $S$ and discuss briefly the first moment of $\calF_{\btau}(T_n^S)$ as $n \to \infty$. 
There are two competing factors contributing to $\calF_{\btau}(T_n^S)$. 
First, due to the "rich gets richer" phenomenon, the "oldest" vertices of $T_n^S$ have degree evolving as $n^{\frac{1}{1+\alpha}}$. %(see~\eqref{convergencedegreesinaffinePAtrees}).
Thus, the contribution to $\bbE[\calF_{\btau}(T_n^S)]$ of embeddings included in the seed (or close to it) 
is of the order $n^{\frac{|\ell|}{1+\alpha}}$ for $|\ell| =\sum_{v\in \tau} \ell(v)$.
% (to be precise, one has to add to the sum of the decorations the number of leaves $v$ of $\tau$ with $\ell(v) = 0$). 

Second, one should take into account the embeddings using recently acquired vertices. 
There are many such vertices, but they have small degrees. 
Depending on the form of $\btau$, such embeddings contribute to $\bbE[\calF_{\btau}(T_n^S)]$ by 
a quantity that may exceed $n^{\frac{|\ell|}{1+\alpha}}$ by some logarithmic or even polynomial factor. 
Thus we find
\begin{align}\label{eq:first_moment1}
	\mathbb{E}[\mathcal{F}_{\btau}(T_{n})] \approx \log^{\gamma(\btau)}(n) \cdot n^{\frac{\lambda(\btau)}{1+\alpha}},
\end{align}
for some $\gamma(\btau) \geq 0$ and $\lambda(\btau) \geq |\ell|$. 
Both values $\gamma(\btau)$ and $\lambda(\btau)$ are computed explicitly in Section~\ref{sec:moments}.

Due to the coupling between $\alpha$-PA started from different seeds (see Section~\ref{sec:planar_PA}), only the first type of embeddings is  sensitive to the seed. 
Thus, for two distinct seeds $S,S'$, we may expect that 
\begin{align}\label{eq:first_moment_difference}
	\mathbb{E}[\mathcal{F}_{\btau}(T_{n}^{S})] - \mathbb{E}[\mathcal{F}_{\btau}(T_{n}^{S'})]
	\approx n^{\frac{|\ell|}{1+\alpha}},
\end{align}
provided that a difference exists. 
Let us, for the sake of this explanation, assume that the standard deviations 
of $\mathcal{F}_{\btau}(T_{n}^{S})$ and $\mathcal{F}_{\btau}(T_{n}^{S'})$
are of the same order as their expectations. 
Then, in order to successfully apply Lemma~\ref{lem:d_tv_moments}, we should have $\gamma = 0$ and $\lambda(\btau) = |\ell|$. 

Such asymptotics (slightly weaker actually) were already observed in \cite{curien2015scaling} for $\alpha=1$; 
they were proved using an amenable recursive relation for $\mathbb{E}[\mathcal{F}_{\btau}(T_{n})]$. 
Rather than showing that $\btau$ may be chosen to satisfy~\eqref{eq:first_moment_difference} 
and so that $\gamma(\btau) = 0$ and $\lambda(\btau) = |\ell|$, 
the authors of \cite{curien2015scaling} constructed a linear combination over trees $\btau$ of observables $\mathcal{F}_{\btau}(T_{n})$ for which the logarithmic factors cancel out. The resulting observable, properly rescaled, turns out to be a martingale that is bounded in $L^2$ and whose expectation depends on the seed tree. 

We will employ a different, arguably simpler strategy: we will prove that for any two distinct seeds $S\neq S'$, the decorated tree $\btau$ may be chosen so as to observe a difference as in~\eqref{eq:first_moment_difference} and such that $\gamma(\btau) = 0$ and $\lambda(\btau) = |\ell|$.
As such, our strategy is more likely to apply to other attachment mechanisms, as will be discussed in Section~\ref{sec:open_problems}.
Additional differences with \cite{curien2015scaling}, are the recurrence formula used to prove~\eqref{eq:first_moment1}, which is more complex
in the affine case, and the fact that the exponent $\lambda(\btau)$ is not always equal to $|\ell|$, as opposed to the linear case, where $\lambda(\btau) = |\ell|$ always.

Let us finally mention that the exact definition \eqref{eq:observable} of the observable is somewhat arbitrary. 
Indeed, it is also possible to use slight variations instead, such as 
$\sum_{\phi: \tau \hookrightarrow T} \prod_{u \in \tau} (\deg_{T}\phi(u))^{\ell(u)}$.
We chose \eqref{eq:observable} as it is inspired by the equivalent construction in \cite{curien2015scaling}
and has an interpretation in the planar version of the model (see Section \ref{sec:planar_PA}).

\paragraph{Organisation of the paper}
In Section~\ref{sec:planar_PA}, we introduce a planar version of the affine preferential attachment model. This construction is not formally necessary to prove our main result, but we believe it is interesting in its own right and helps clarify the subsequent proofs. 
In particular, the planar version will give rise to a natural coupling between two $\alpha$-PA trees starting from two different seeds of same size. 

In Section~\ref{sec:moments} we study the asymptotics of the first and second moments of $\mathcal{F}_{\btau}(T_n^S)$ as $n\to \infty$. The only role of this section within the proof of Theorem~\ref{thm:main} is to show a bound on the second moment of $\calF_{\btau}(T_n^S)$ for certain trees $\btau$.
These estimates use a recurrence relation for $\mathcal{F}_{\btau}(T_n^S)$, which is similar to that of \cite{curien2015scaling}, but more complicated due to the affine probabilities.

Section~\ref{observablesandtheirdifferencearoundthegraphseed} contains lower bounds on the first-moment difference of $\mathcal{F}_{\btau}(T_n^{S})$ for two distinct seeds $S$. Precisely, we prove~\eqref{eq:first_moment_difference} for well-chosen decorated trees $\btau$. 
This is then used to prove Theorem~\ref{thm:main}. 

Finally, certain variations of the model and potentiel extensions of our result are discussed in Section~\ref{sec:open_problems}.

\paragraph{Notation}
In the rest of the paper, we will use the following notations:
\begin{tight_itemize}\itemsep 0pt
    \item for two sequences $f,g:\bbN \to (0,+\infty)$, 
    write $f(n) \approx g(n)$ if there exists some constant  $C > 0$ such that  $\frac1C \leq \frac{f(n)}{g(n)} \leq C$; 
    \item also $f(n) = \mathcal{O}\big(g(n)\big)$ if $\limsup\limits_{n \to +\infty} \ \lvert \frac{f(n)}{g(n)} \rvert < +\infty$;
    \item and $f(n) \ll g(n)$ if $\lim\limits_{n \to +\infty} \lvert \frac{f(n)}{g(n)} \rvert = 0$.
    \item The size of a tree $ T$ is written $|T|$ and stands for the number of vertices of $ T$. 
    \item For a graph $S$, write $V_S$ and $E_S$ for its sets of vertices and edges, respectively. 
    For $v \in V_S$, write $\deg_S(v)$ or $\deg(v)$ for the degree of $v$ in $S$. 
\end{tight_itemize}

\paragraph{Acknowledgements} 
The authors thank Nicolas Curien for helpful discussions and corrections on an earlier version of this paper. 

The first author is supported by a PhD Fellowship from the MESR, France and was supported in the early stages of this project by the NCCR SwissMAP.  
The second author is a member of the NCCR SwissMAP.  

%%%%%%%%%%%%% END OF THE INTRODUCTION %%%%%%%%%%%%%%%%% %%%%%%%%%%%%%%%%%%%%%%%%%%%%%%%%%%%%%%%%%%%%%%%
%%%%%%%%%%%%%%%%%%%%%%%%%%%%%%%%%%%%%%%%%%%%%%%%
%%%%%%%%%%%%%%%%%%%%%%%%%%%%%%%%%%%%%%%%%%%%%%%
%%%%%%%%%%%%%%%%%%%%%%%%%%%%%%%%%%%%%%%%%%%%%%%%

\section{Planar affine preferential attachment model}\label{sec:planar_PA}

In this section, we first define the planar version of affine preferential attachment trees,
then we present a useful coupling between any two $\alpha$-PA trees derived from different seed trees of same size. 
The coupling is based on the decomposition of $T_n^S$ into \textit{planted plane trees}, 
which isolates the roles of the seed and of the growth mechanism, respectively. 

We reiterate that this planar version is not formally necessary for any of the results; it is used merely to illustrate the arguments. 

\subsection{Definition via corners}\label{sec:corners}

We first need to introduce the notion of \textit{colouring of corners} in a plane tree. A plane tree is a tree embedded in the plane up to continuous deformation, or equivalently a tree whose vertices are equipped with a cyclic order of their neighbours.
A \textit{corner} of a plane tree is an angular sector of the plane contained between two consecutive half-edges around a vertex. 
In particular, the number of corners surrounding a vertex is equal to its degree.

\begin{definition}[Colouring of corners]\label{defblueredcoins}
    A colouring of the corners of a plane tree $\mathcal T$ is a choice of one distinguished corner for each vertex of the tree. 
    The distinguished corners are called ``red''; all other corners are called ``blue''.
    
	When indexing $\{c_{v,i}:\, v\in V_{ T}, 1 \leq i \leq \deg(v_i)\}$ the corners of $\mathcal T$,
	we will always suppose that the red vertices are $(c_{v,1})_{v \in V_{\mathcal T}}$. 
\end{definition}

It will be useful to keep in mind that any plane tree with $n$ vertices and coloured corners has $n$ red corners and $n-2$ blue ones. 
We are now ready to introduce the planar version of affine preferential attachment model. 

\begin{definition}
Fix a parameter $\alpha > 0$ and a plane tree $S$ - namely the \textit{seed tree} - of size $\vert S \vert = k \geq 2$ endowed with a colouring of its corners. The planar $\alpha$-PA with seed $S$ is a growing sequence of random plane trees $(T_{n}^{S})_{n \geq k}$ (or $(T_{n})_{n \geq k}$ when the role of $S$ is clear) 
with coloured corners, evolving according to the following inductive principle:
\begin{tight_itemize}
	\item[(i)] $T_{k} = S$;
	\item[(ii)] Suppose that our sequence of plane trees is built until step $n \geq k$. 
    Then, independently from the past of the process, select at random a corner $c_n$ of $T_n$ with 
    \begin{align*}
    \bbP (c_n = c_{u,1}) = \frac{\alpha}{n  (1+\alpha) - 2} 
    \quad \text{ and } \quad  \bbP (c_n = c_{u,i}) = \frac{1}{n  (1+\alpha) - 2} \,\, \, \forall 2 \leq i \leq \deg_{T_n} (u),
    \end{align*}
    for any vertex $u \in V_{T_n}$.
    That is, red corners are chosen with probability proportional to $\alpha$ and blue corners with probability proportional to $1$.
    \item[(iii)] Write $u_n$ for the vertex of $T_n$ of which $c_n$ is a corner. 
    To obtain $T_{n+1}$ add a new vertex $v_{n+1}$ to $T_n$ and we connect it by an edge to $u_n$ through the corner $c_n$. 
    \item[(iv)] All corners of $T_n$ except $c_n$ subsist in $T_{n+1}$ and conserve their colour. 
    The unique corner $c_n''$ of the new vertex $v_{n+1}$ is coloured red, as required by Definition~\ref{defblueredcoins}. 
    The randomly selected corner $c_n$ is divided in two angular areas by the edge connecting $u_n$ to $v_{n+1}$. 
    If $c_n$ was blue in $T_n$, then both these corners of $T_{n+1}$ are blue. 
    If $c_n$ was red in $T_n$, we arbitrarily colour the corner to the right of the edge $(u_n,v_{n+1})$ in red and the other in blue. 
   	\end{tight_itemize}
\end{definition}

The probabilities appearing in step (ii) indeed sum to 1 since:
\begin{align}
	\underbrace{n}_{\textrm{number of red corners}} \cdot \ \frac{\alpha}{n \cdot (1+\alpha) - 2} + \underbrace{(n-2)}_{\textrm{number of blue corners}} \cdot \ \frac{1}{n \cdot (1+\alpha) - 2} =1 \nonumber
\end{align}

Observe that, the attachment principle in the planar $\alpha$-PA is such that the probability for $v_{n+1}$ to be attached to a vertex $u$ of $T_n$ is $\frac{\deg_{T_n}(u)-1+\alpha}{n \cdot (1+\alpha)-2}$. Thus, we obtain :

\begin{lemma}
	If $(T_n)_{n \geq k}$ is a planar $\alpha$-PA started from some plane seed tree $S$, 
	then the trees $(T_n)_{n \geq k}$ (stripped of their planar embedding) have the law of a $\alpha$-PA. 
\end{lemma}

\begin{remark}\label{initialcolouredlabellingisarbitraryandhasnoimportance}
    Changing the rules of the assignment of labels and colours to the new corners of a planar $\alpha$-PA 
    affects the law of the plane tree, 
    but only via its embedding. 
    
    In the special case $\alpha = 1$, the weights of red and blue corners are equal, 
    and the planar $\alpha$-PA is simply a uniform embedding of the abstract $\alpha$-PA. 
    This is not the case when $\alpha \neq 1$. For more on this see \cite[Conj. 2]{curien2015scaling}.
\end{remark}

As mentioned in the introduction, the observables $\calF_{\btau}$ have a particular interpretation in terms of the planar model. 
Let $\btau = (\tau,\ell)$ be a decorated tree and $T$ be a plane tree with coloured corners. 
Then, for any embedding $\phi:\tau \hookrightarrow T$, the factor $\prod_{u \in \tau} [\deg_{T}(\phi(u))-1]_{\ell(u)}$
is the number of ordered choices of $\ell(u)$ different blue corners around each vertex $\phi(u)$ for $u\in \tau$. 
As in \cite{curien2015scaling}, one may imagine that each vertex $u \in \tau$ is endowed with $\ell(u)$ distinct arrows. 
Call a decorated embedding an embedding of $\phi$ of $\tau$ in $T$ together with, for each $u \in \tau$ and each arrow of $u$, 
a blue corner of $\phi(u)$ to which that arrow points, in such a way that distinct arrows point to distinct corners. 
Then $\calF_{\btau}(T)$ is the number of decorated embeddings of $\tau$ in $T$.

\subsection{Decomposition and coupling using planted plane trees}\label{sec:coupling}

Apart from its intrinsic interest, the decomposition described below will be used in Section~\ref{observablesandtheirdifferencearoundthegraphseed}. 
We begin with the definition of a planted plane tree:

\begin{definition}\label{definitionplanarplantedtree}
	A planted plane tree $T_{\multimap}$ is a plane tree $T$ 
	with a distinguished vertex called the root and an additional half-edge emerging from the root. 
	This half edge divides the corner of the root 
	delimited by the two edges linking the root to its leftmost and its rightmost children. 
	When $\lvert T\rvert=1$, the planted plane tree $T_{\multimap}$ is merely a single vertex with an half-edge attached to it.
\end{definition} 

It should be noted that there is one more corner in any planted version $T_{\multimap}$ of a plane tree $T$.
A colouring of corners for a planted plane tree $T_{\multimap}$ is defined as for plane trees (see Definition~\ref{defblueredcoins}) with the exception that the root is allowed to have no red corners or one red corner, which will always be the corner to the right of the half-edge. 
In the former situation, we say that $T_{\multimap}$ is a \textit{\textbf{blue planted plane tree}}, in the latter we say it is a \textit{\textbf{red planted plane tree}}. 

Since the recursive procedure used to define the planar $\alpha$-PA tree is simply based on the colouring of corners, 
we straightforwardly adapt it to define a planted planar version of the same model. 
Note by $T_{n}^{\multimap b}$, resp. $T_{n}^{\multimap r}$, a blue planted plane tree, resp. a red planted plane tree, obtained through the preferential attachment algorithm with initial condition $\multimap$, which is the tree composed of a single vertex with a half-edge attached to it and a blue corner, resp. a red corner, surrounding it.

Let $S$ be a plane seed tree of size $k \geq 2$ with coloured corners indexed $\{c_{v,i}:\, v \in V_S, 1 \leq i \leq \deg_S(v)\}$. 
Fix $n \geq k$ and $T_{n}$ be a realisation of the planar $\alpha$-PA tree at step $n$ starting from~$S$. 
For~$v \in V_S$ and $1 \leq i \leq \deg_S(v)$, say that a vertex $u \in V_{T_n} \setminus V_S$ is a descendent of the corner~$c_{v,i}$ 
if the unique path linking $u$ to $S$ arrives at $S$ through $c_{v,i}$. 
Additionally, say that $v$ is also a descendent of $c_{v,i}$. 
The tree $T_n^{v,i}$ stemming from $c_{v,i}$ is the planted plane tree
containing all vertices of $T_n$ that are descendent of $c_{v,i}$
and all edges of $T_n$ between such vertices; it is rooted at~$v$. 
The half-edge attached to $v$ is such that it does not split the corner $c_{v,i}$. 
Call these trees the planted plane subtrees of $T_n^S$. 
See Figure~\ref{decompositionplantedplanartreefig} for an illustration.

\begin{proposition}\label{fundamentaldecompositionplantedplanartrees}
    Let $S$ be a plane seed tree of size $k \geq 2$ with coloured corners indexed $\{c_{v,i}:\, v \in V_S, 1 \leq i \leq \deg_S(v)\}$. 
    For $n\geq k$, let $T_{n}^{v,i}$ be the planted plane tree stemming from $c_{v,i}$ and let $k_n^{v,i}$ denote its size.
    Write $x_{n}^{v,1} = (1+\alpha)k_n^{v,1} - 1$ and $x_{n}^{v,i} = (1+\alpha)k_n^{v,i} - \alpha$ for all $v \in V_S$ and $i \geq 2$. 
    Then 
    \begin{tight_itemize}
    \item  the vector $(x_n^{v,i})_{v \in V_S, 1 \leq i \leq \deg_S(v)}$ has the distribution of a P\'olya urn with $2k-2$ colours and
    diagonal replacement matrix $(1+\alpha) I_{2k-2}$, starting from $(\alpha \ind_{\{i =1\}} + \ind_{\{i\neq 1\}})_{v \in V_S, 1 \leq i \leq \deg_S(v)}$;
    \item conditionally on $(x_n^{v,i})_{v \in V_S, 1 \leq i \leq \deg_S(v)}$, the trees $(T_{n}^{v,i}))_{v \in V_S, 1 \leq i \leq \deg_S(v)}$ are independent, with $T_{n}^{v,i}$ having law $T_{k_n^{v,i}}^{\multimap r}$ if $i =1$ and $T_{k_n^{v,i}}^{\multimap b}$ if $i \neq 1$. 
    \end{tight_itemize} 
\end{proposition}

Proposition~\ref{fundamentaldecompositionplantedplanartrees} may be restated as follows. 
Given a plane seed $S$ with $|S| = k$, we would like to construct $T_n^S$ for some $n \geq k$. 
This may be done in the following steps. 
\begin{enumerate}
	\item \label{step1conversetheoremdecompositionplanted} 
	Generate a vector $(x_{n}^{v,i})_{v \in V_S, 1 \leq i\leq \deg_S(v)} \in \mathbb{R}^{2k-2}$ 
	with the law of the P\'olya urn of Proposition~\ref{fundamentaldecompositionplantedplanartrees}.
	Define sizes $k_{n}^{v,1}:= \frac{x_{\ell,p}^{(b)}+\alpha}{1+\alpha}$ and $k_{n}^{v,i}:= \frac{x_{\ell}^{(r)}+1}{1+\alpha}$ 
	for $v \in V_S$ and $i \geq 2$. 
	\item Randomly draw independent realisations $\tau_n^{v,1}$ of $T_{k_{n}^{v,1}}^{\multimap r}$ and 
	$\tau_n^{v,i}$ of $T_{k_{n}^{v,i}}^{\multimap b}$ for each $v \in V_S$ and~$i \geq 2$.
	\item \label{step4conversetheoremdecompositionplanted} 
	Graft each tree $\tau_n^{v,i}$ in the corner $c_{v,i}$ of $S$.
\end{enumerate}
The resulting tree has the law of a planar $\alpha$-PA started from $S$. \\

This decomposition allows to couple the evolution of two planar $\alpha$-PA trees emerging from distincts seeds of same size.
Indeed, in spite of the notation, the first two steps do not depend on $S$, only on $k = |S|$: 
in the first step, the vector $(x_{n}^{v,i})$ 
contains $2k-2$ entries and the P\'olya urn generating it starts with $k$ entries $\alpha$ and $k-2$ entries $1$; 
in the second step, $k$ trees are of type $T^{\multimap r}$ and $k-2$ of type $T^{\multimap b}$.
If $S$, $S'$ are two seeds of same size $k$, then $T_n^{S}$ and $T_n^{S'}$ are coupled as follows. 
Simulate a common vector for the sizes of the planted plane subtrees of $T_n^{S}$ and $T_n^{S'}$, 
then simulate a common set of planted plane trees. 
Finally graft these trees to $S$ to obtain $T_n^{S}$ and the same trees to $S'$ to obtain $T_n^{S'}$.

\begin{figure}
    \centering 
    \includegraphics[scale=0.5]{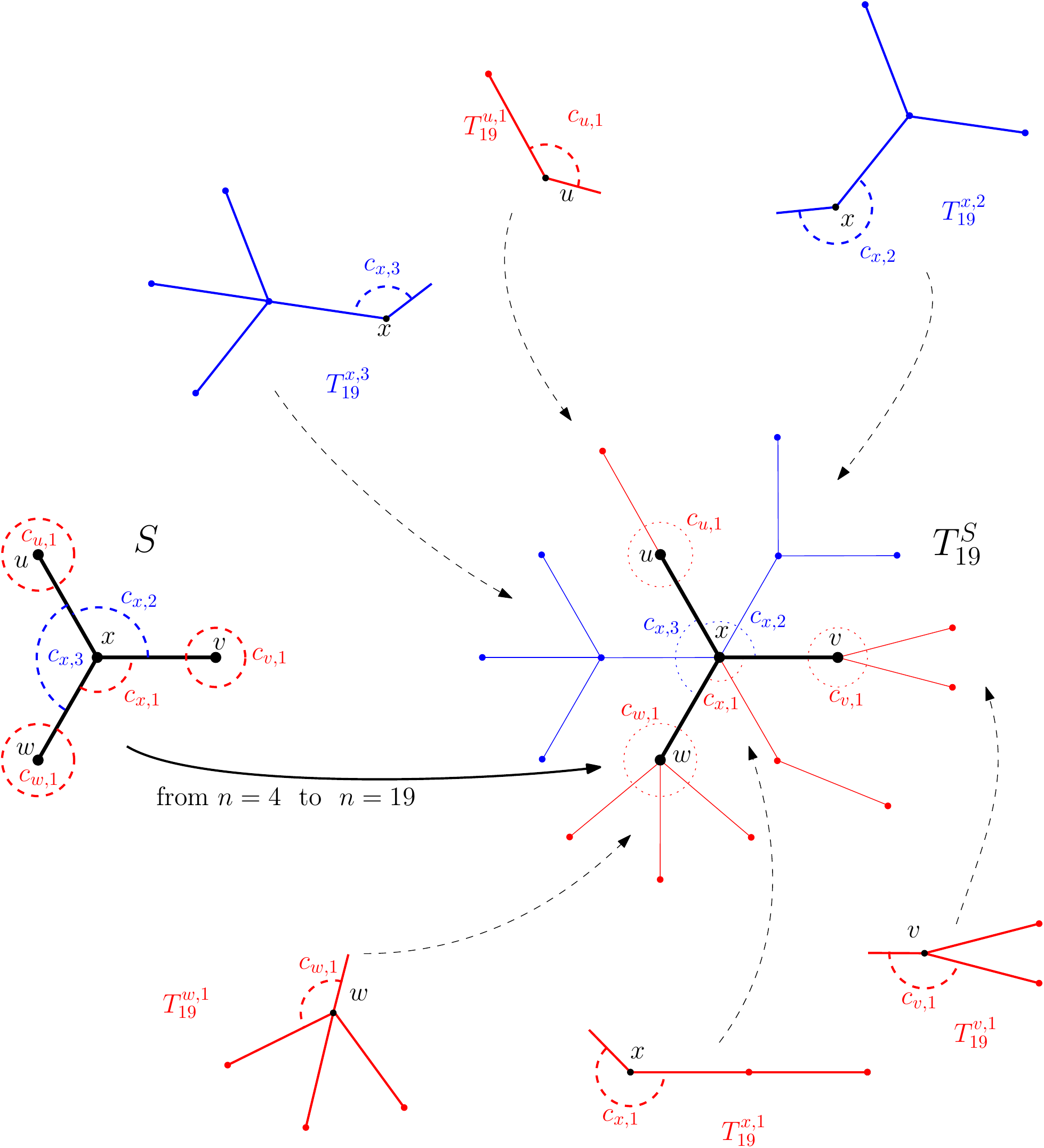}
    \caption{An example of decomposition into red and blue planted plane tree.}
    \label{decompositionplantedplanartreefig}
\end{figure}

\begin{proof}[Proof of Proposition~\ref{fundamentaldecompositionplantedplanartrees}]
	Fix $S$. Let us fist prove that the vector $(x_n^{v,i})_{v \in V_S, 1 \leq i \leq \deg_S(v)}$ has the distribution of the P\'olya urn described in the statement. We do this by induction on $n$.
	When $n =k$, all planted plane subtrees of $T_k^S = S$ have size $1$. 
	The formula relating the sizes of the subtrees to $x_n^{v,i}$ then yields $x_n^{v,1} = \alpha$ and $x_{n}^{v,i} = 1$ for $i \geq 2$. 
	
	Suppose now that the statement is proved up to step $n$ and let $T_n^S$ be a realisation of the planar $\alpha$-PA tree of size $n$. 
	Then, each red planted plane subtree $T_{n}^{v,1}$ has $k_n^{v,1}$ red corners and $k_n^{v,1}-1$ blue ones. 
	Each blue planted plane subtree $T_n^{v,i}$ with $i \geq 2$ has $k_n^{v,i}-1$ red corners and $k_n^{v,i}$ blue corners.
	It follows that, for each $v$
	\begin{align*}
		\bbP(v_{n+1} \text{ descendent of } c_{v,1}) & = \frac{(\alpha +1)k_n^{v,1}-1}{2n-2} = \frac{x_{n}^{v,1}}{2n-2} \qquad \text{and}\\
		\bbP(v_{n+1} \text{ descendent of } c_{v,i}) & = \frac{(\alpha +1)k_n^{v,1}-\alpha}{2n-2} =\frac{x_{n}^{v,i}}{2n-2}
		 \qquad \text{for all $i \geq 2$.}	
	\end{align*}
	Finally, if $v_{n+1}$ is attached to some descendant of a corner $c_{v,i}$, then $x_{n+1}^{v,i} = 1 + \alpha + x_{n}^{v,i}$. 
	This shows that the vector $(x_{n+1}^{v,i})$ also has the claimed distribution. 
	
	Let us now condition on $(x_{n}^{v,i})_{v,i}$, 
	or equivalently on $(k_{n}^{v,i})_{v,i}$ for some $n$. 
	Then, for each $v \in V_s$ and $i \leq \deg_S(v)$, 
	the $k_{n}^{v,i}$ vertices of $T_n^{v,i}$, independently of the times they join $T_k,\dots, T_n$, 
	attach themselves to the planted plane subtree of the corner $c_{v,i}$ with a (conditional) distribution
	that assigns to any red corner a weight proportional to $\alpha$ and to any blue corner a weight proportional to~$1$. 
	If follows readily that $T_{n}^{v,i}$ has the law of $T_{k_n^{v,i}}^{\multimap r}$ or $T_{k_n^{v,i}}^{\multimap b}$, 
	depending on whether $c_{v,i}$ is red or blue, respectively. 
	It is also immediate, that the resulting trees in different corners are independent%
	\footnote{This independence only holds {\em conditionally on the size vector} $(k_{n}^{v,i})_{v,i}$. Moreover, we do not claim that the evolution between step $k$ and $n$ is independent in different corners, only the resulting subtrees are.}.
\end{proof}

%%%%%%%%%%%%%%%% END OF THE SECTION "PLANAR MODEL AND A COUPLING"%%%% %%%%%%%%%%%%%%%%%%%%%%%%%%%%%%%%%%%%%%%%%%%%%%%%%
%%%%%%%%%%%%%%%%%%%%%%%%%%%%%%%%%%%%%%%%%%%%%%%%%
%%%%%%%%%%%%%%%%%%%%%%%%%%%%%%%%%%%%%%%%%%%%%%%%%
%%%%%%%%%%%%%%%%%%%%%%%%%%%%%%%%%%%%%%%%%%%%%%%%%%%

\section{First and second moment of a class of observables}\label{sec:moments}

The ultimate goal of this section is to obtain precise estimates on the second moment of $\mathcal{F}_{\btau}(T_n)$ 
for particular decorated trees $\btau$.
In doing so we will also prove a general result on the first moment of $\mathcal{F}_{\btau}(T_n)$ for any $\btau$. 
The latter is longer to state, and is deferred to later in the section. 
Below is the minimal result required in the proof of Theorem~\ref{thm:main}. 

\begin{theorem}\label{thm:second_moment}
    Let $\btau$ be a finite decorated tree with $\ell(u) \geq 2$ for any $u \in \tau$
    and such that $|\ell| = \sum_{u \in \tau} \ell(u) > 1 + \alpha$. 
    For any seed $S$, we have the following asymptotic as $n \to \infty$: 
    \begin{align*}
	    \mathbb{E}[\mathcal{F}_{\btau}(T_{n}^S)^2] \approx n^{\frac{2 |\ell|}{1+\alpha}}.
    \end{align*}
\end{theorem}

To prove the above, we will proceed in several steps, each occupying a subsection below. 
First we prove a recurrence relation on the first moment of our observables (see Section~\ref{sec:reucrrence}),
which is then used in Section~\ref{sec:first_moment} to obtain an accurate estimate of 
$\mathbb{E}[\mathcal{F}_{\btau}(T_{n}^S)]$ as $n$ goes to infinity (see Theorem~\ref{thm:first_moment}). 
Finally, in Section~\ref{secondmomentestimatetheoremsection}, we use the previous results to prove Theorem~\ref{thm:second_moment}.

\subsection{A recurrence formula}\label{sec:reucrrence}

Recall from the introduction that the asymptotic of $\bbE[\calF_{\btau}(T_n^S)]$ as $n\to \infty$ is polynomial with an exponent depending on $\btau$. 
This exponent will be determined by the weight of $\btau$ defined below.

\begin{definition}[Weight of decorated trees]\label{effectiveweight}
    Let $\btau$ be a decorated tree.
    If $\tau$ has size $1$ or $2$ and all decorations of its vertices are $0$, then the weight of $\btau$ is $1$. 
    In all other cases, the weight of $\btau$ is 
    \begin{align*}%\label{effective weight}
    	w(\btau) =\sum_{v\in V_\tau} \ell(v) + \ind_{\{\deg_\tau(v) = 1 \text{ and } \ell(v) = 0\}}.
    \end{align*}
    Call a vertex $v \in \tau$ with $\ell(v) = 0$ and $\deg_\tau(v) = 1$ a {\em loose leaf} of $\btau$. 
    Then the weight of $\btau$ is $|\ell|$ plus the number of loose leaves.
\end{definition}

Definition~\ref{effectiveweight} induces a partial order on the set of decorated trees.

\begin{definition}[Partial order on decorated trees]\label{partialorderdef}
    For any two decorated trees $\btau = (\tau,\ell)$, $\btau' = (\tau',\ell')$
    we write $\boldsymbol{\tau'} \prec \btau$ if and only if:
    \begin{tight_itemize}
        \item $w(\boldsymbol{\tau'}) < w(\btau)$ and $\vert {\tau'} \vert \leq \vert {\tau} \vert$;
        \item or if $w(\boldsymbol{\tau'}) = w(\btau)$ and $\vert {\tau'} \vert < \vert {\tau} \vert$; 
        \item or if $w(\boldsymbol{\tau'}) = w(\btau)$, $\vert {\tau'} \vert = \vert {\tau} \vert$ and 
        $\sum_{v\in V_{\tau'}} \ell'(v) = |\ell'| < |\ell| = \sum_{v\in V_\tau} \ell(v)$.
    \end{tight_itemize}
\end{definition}

Clearly, $\prec$ is a strict partial order on the set of decorated trees. We denote by $\preceq$ the associated partial order. 

There are three trees $\btau$ of weight $1$: those with a single vertex and decoration $0$ or $1$
(they are denoted by $\small{\textcircled{\tiny{0}}}$ and $\small{\textcircled{\tiny{1}}}$, respectively)
and 
that with two vertices and decorations $0$ for both of them (denoted by $\small{\textcircled{\tiny{0}}}-\small{\textcircled{\tiny{0}}}$).

\newcommand{\ct}{\mathsf{c}}

\begin{proposition}\label{prop:recurrence}
    There exists a family of nonnegative real numbers 
    $\{\ct(\bsigma,\btau) \: \ \bsigma \prec\btau \}$ 
    such that, for every seed $S$, every decorated tree $\btau$ with $w(\btau) >1$ and every $n \geq |S|$:
    \begin{align}\label{eq:recurrence}
        \bbE[\calF_{\btau}(T_{n+1}^{S}) \vert \calE_n] 
        = \Bigg(1+\frac{w(\btau)}{(1+\alpha) n-2}\Bigg)\cdot \calF_{\btau}(T_{n}^{S}) 
        + \frac{1}{(1+\alpha) n-2} \sum_{\bsigma \prec \btau}  \ct(\bsigma,\btau) \cdot \calF_{\bsigma}(T_{n}^{S}) ,
    \end{align}
    where $\calE_n$ is the $\sigma$-algebra generated by $T_{k}^S,\dotsc,T_n^S$. In addition, when: 
    \begin{tight_itemize}
        \item $\btau = \small{\textcircled{\tiny{0}}}$, $\calF_{\btau}(T_{n}) = n$;
        \item $\btau = \small{\textcircled{\tiny{1}}}$, $\calF_{\btau}(T_{n}) = n-2$;
        \item $\btau = \small{\textcircled{\tiny{0}}}-\small{\textcircled{\tiny{0}}}$, $\calF_{\btau}(T_{n}) = 2n-2$.
	\end{tight_itemize}
\end{proposition}

It is worth mentioning that the constants $\{c(\bsigma,\btau) \: \ \bsigma \prec\btau \}$ do not depend on $n$ or $S$, 
but do depend on the parameter $\alpha$ of the model. 

The proof below is somewhat algebraic and does not use the planar $\alpha$-PA. 
A more visual proof that uses the notion of decorated embedding (as described at the end of Section~\ref{sec:corners}) may be given. 
It is similar to that of \cite[Lem. 6]{curien2015scaling} with some additional difficulties due to the colouring of corners. 

\begin{proof}
	Fix the seed tree $S$ and drop it from the notation.
	First let us prove the three particular cases. 
	\begin{tight_itemize}
        \item For $\btau = \small{\textcircled{\tiny{0}}}$, then $\calF_\tau(T_n)$ is simply the number of vertices of $T_n$, hence is equal to $n$.
        \item For $\btau = \small{\textcircled{\tiny{1}}}$, 
        then $\calF_{\btau}(T) = \sum_{u \in T} \deg_T(u) -1  = 2 |E_T| - |V_T| = |V_T| -2$. 
       	Hence $\calF_{\btau}(T_n) = n-2$. 
        \item For $\btau = \small{\textcircled{\tiny{0}}}-\small{\textcircled{\tiny{0}}}$, 
        then $\calF_{\btau}(T) =2|E_T|$ since the $\btau$ may be embedded over any edge of $T$ in one of two directions. 
        Hence $\calF_{\btau}(T_n) = 2n-2$. 	
    \end{tight_itemize}
	
	Let us now prove the recurrence formula.
	Fix $\btau = (\tau,\ell)$ a decorated tree with $w(\btau) >1$. 
	For a tree $T$ and an embedding $\phi: \tau \hookrightarrow T$, let 
	$\pi(\btau,\phi, T) = \prod_{u \in \tau} [\deg_{T}(\phi(u))-1]_{\ell(u)}$, so that 
	\begin{align}\label{eq:observable2}
		\mathcal{F}_{\btau}(T)
		= \sum_{\phi: \tau \hookrightarrow T} \,\prod_{u \in \tau} [\deg_{T}(\phi(u))-1]_{\ell(u)} 
		= \sum_{\phi: \tau \hookrightarrow T} \pi(\btau,\phi, T).
	\end{align}
	
	Recall that, in passing from $T_n$ to $T_{n+1}$, a new vertex $v_n$ is attached to a randomly chosen vertex $u_n$ of $T_n$.
	Our purpose is to compute the sum above over embeddings $\phi$ of $\tau$ in $T_{n+1}$. 
	We may restrict the sum only to embeddings with $\pi(\btau,\phi,T_{n+1}) > 0$. 
	We separate such embeddings into three categories:
	\begin{enumerate}
	\item those who do not include $u_n$ or $v_n$ in their image;
	\item those who include $u_n$ but not $v_n$ in their image;
	\item those who include both $u_n$ and $v_n$ in their image.
	\end{enumerate}
	The embedding cannot contain $v_n$ without $u_n$ in its image. 
	Indeed, we have $\deg_{T_{n+1}}(v_n) = 1$, hence, if $\phi$ is an embedding that maps a vertex $\tv \in \tau$ to $v_n$, then for $\pi(\btau,\phi, T_{n+1})$ to be non-zero, it is necessary that $\ell(\tv) = 0$. 
	By choice of $\tau$, if such a vertex exists, it necessarily has a neighbour $\tu$, which then is mapped by $\phi$ onto $u_n$. 
	Moreover, the vertex $\tv$ needs to have a single neighbour in $\tau$, hence needs to be a loose leaf of $\btau$. 
	
	Write $\calF^{(i)}_{\btau}(T_{n+1})$ with $i=1,2,3$ for the contribution to~\eqref{eq:observable2} of embeddings from each of the categories above. 
	For the first two categories, the embeddings considered are in one to one correspondence with embedding of $\tau$ in $T_n$ (although their weights are different whether considered in $T_n$ or $T_{n+1}$). 
	Thus 
	\begin{align*}
		\calF^{(1)}_{\btau}(T_{n+1}) 
		&= \sum_{\phi: \tau \hookrightarrow T_n} \pi(\btau,\phi, T_n) \cdot \prod_{\tu \in \tau}\ind_{\phi(\tu) \neq u_n} \qquad \text{ and }\\
		\calF^{(2)}_{\btau}(T_{n+1}) 
		&= \sum_{\phi: \tau \hookrightarrow T_n}\sum_{\tu \in \tau} 
		\ind_{\phi(\tu) = u_n}\Big( \prod_{\tv \in \tau \setminus \{\tu\}} [\deg_{T_n}\phi(\tv)-1]_{\ell(\tv)} \Big) 
		 \cdot [\deg_{T_n}\phi(\tu)]_{\ell(\tu)}.
	\end{align*}
	Now a basic algebraic manipulation shows that $[k]_\ell = [k-1]_{\ell} + \ell [k-1]_{\ell -1}$ 
	for all $k \geq 1$ and $\ell \geq 0$\footnote{when $\ell = 0$, the second term in the RHS is not defined.
	However, its prefactor cancels it out, and we allow this abuse of notation.}.
	Moreover $\prod_{\tu \in \tau}\ind_{\phi(\tu) \neq u_n} = 1 - \sum_{\tu \in \tau} \ind_{\phi(\tu) = u_n}$. Hence
	\begin{align*}
		&\calF^{(1)}_{\btau}(T_{n+1}) +\calF^{(2)}_{\btau}(T_{n+1})\\
		&= \sum_{\phi: \tau \hookrightarrow T_n}
		\Big[ \pi(\btau,\phi, T_n) 
		+ \sum_{\tu \in \tau} 
		\ind_{\phi(\tu) = u_n}
		\ell(\tu) \Big( \prod_{\tv \in \tau \setminus \{\tu\}} [\deg_{T_n}\phi(\tv)-1]_{\ell(\tv)} \Big) 
		 \cdot [\deg_{T_n}\phi(\tu)-1]_{\ell(\tu)-1}\Big].
	\end{align*}
	Notice that $\sum_{\phi: \tau \hookrightarrow T_n} \pi(\btau,\phi, T_n) = \calF_{\btau}(T_{n})$. 
	Moreover, recall that, for any fixed $\tu \in \tau$ and embedding $\phi: \tau \hookrightarrow T_n$, we have 
	$\bbP(\phi(\tu) = u_n\,|\, \calE_n) = \frac{\deg_{T_n}\phi(\tu)-1+\alpha}{(1+\alpha)n-2}.$
	Thus, when taking the expectation in the above we find
	\begin{align*}
		& \bbE \big[ \calF^{(1)}_{\btau}(T_{n+1}) +\calF^{(2)}_{\btau}(T_{n+1})\, \big|\, \calE_n\big]- 
		\calF_{\btau}(T_{n}) \\
		& = \sum_{\tu \in \tau} \ell(\tu) \cdot 
		 \sum_{\phi: \tau \hookrightarrow T_n}
		 \Big( \prod_{\tv \in \tau \setminus \{\tu\}} [\deg_{T_n}\phi(\tv)-1]_{\ell(\tv)} \Big) 
		 \cdot [\deg_{T_n}\phi(\tu)-1]_{\ell(\tu)-1}
		 \cdot\tfrac{\deg_{T_n}\phi(\tu)-1+\alpha}{(1+\alpha)n-2}.
	\end{align*}
	The sum over $\tu \in \tau$ in the right-hand side above may be limited to vertices $\tu$ with $\ell(\tu) > 0$. 
	
	For $\tu \in \tau$ with $\ell(\tu) >0$, let $\btau^{(\tu-)}$ be the decorated tree $(\tau, \ell^{(\tu-)})$ 
	with decorations identical to those of $\btau$ except at the vertex $\tu$ for which $\ell^{(\tu-)} = \ell(\tu) -1$. 
	Then $\btau^{(\tu-)} \prec \btau$. 
	Write $\deg_{T_n}\phi(\tu)-1+\alpha = [\deg_{T_n}\phi(\tu) -\ell(\tu)] + \ell(\tu) +\alpha -1$ 
	to find that the summand in the right hand side is 
	$$\tfrac{1}{(1+\alpha)n-2} \big[ \pi(\btau, \phi, T_n) + (\ell(\tu) +\alpha -1)\cdot \pi(\btau^{\tu-}, \phi, T_n)\big].$$
	Thus we find 
	\begin{align}
		 \bbE \big[ \calF^{(1)}_{\btau}(T_{n+1}) +\calF^{(2)}_{\btau}(T_{n+1})\, \big|\, \calE_n\big] 
		& = \calF_{\btau}(T_{n}) +
%		& = \sum_{\tu \in \tau} \tfrac{\ell(\tu)}{(1+\alpha)n-2} \cdot 
%		 \sum_{\phi: \tau \hookrightarrow T_n} \big[ \pi(\btau, \phi, T_n) + (\ell(\tu) +\alpha -1)\cdot \pi(\btau^{\tu-}, \phi, T_n)\big]\\
		 \sum_{\tu \in \tau} \tfrac{\ell(\tu)}{(1+\alpha)n-2} \cdot 
		\big[\calF_{\btau}(T_n) + (\ell(\tu) +\alpha -1)\cdot  \calF_{\btau^{\tu-}}(T_n)\big] \nonumber \\
		& =
		\Big[1 + \tfrac{|\ell|}{(1+\alpha)n-2}\Big] \calF_{\btau}(T_{n}) +
		 \sum_{\tu \in \tau} \tfrac{\ell(\tu)(\ell(\tu) +\alpha -1)}{(1+\alpha)n-2}\cdot  \calF_{\btau^{\tu-}}(T_n).
		 \label{eq:F12}
	\end{align}
	
	Finally we turn to $\calF^{(3)}_{\btau}(T_{n+1})$. 
	Let $\phi: \tau \hookrightarrow T_{n+1}$ be an embedding contributing to $\calF^{(3)}_{\btau}(T_{n+1})$, let $\tv$ be the loose leaf mapped to $v_{n+1}$ and $\tu$ its only neighbour in $\tau$. 
	Define the following three modifications of $\btau$: 
	\begin{tight_itemize}
	\item $\btau \setminus \tv$ is the decorated tree obtained from $\btau$ by removing the leaf $\tv$ and conserving the same decorations for all other vertices;
	\item $(\btau \setminus \tv)^-$ is the decorated tree obtained from $\btau$ by removing the leaf $\tv$, 
	decreasing the decoration of $\tu$ by one, and conserving the same decorations for all other vertices\footnote{This is only defined when $\ell(\tu) >0$; it will implicitly only appear in such cases in the upcoming formulas};
	\item $(\btau \setminus \tv)^+$ is the decorated tree obtained from $\btau$ by removing the leaf $\tv$, 
	increasing the decoration of $\tu$ by one, and conserving the same decorations for all other vertices.
	\end{tight_itemize}
	It is immediate to check that all trees above are smaller than $\btau$ for the order $\prec$.
	
	Write $\tau \setminus \tv$ for the tree (stripped of decoration) of all of the above. 
	To $\phi$ associate its restriction $\tilde\phi : \tau \setminus\tv \to T_n$ to $\tau \setminus\tv$. 
	Then, by the same type of computation as above
	\begin{align}\label{eq:piv}
		\pi(\btau, \phi,T_{n+1})
		& = \Big( \prod_{\tw \in \tau \setminus \tv} [\deg_{T_{n+1}}\phi(\tw)-1]_{\ell(\tw)} \Big)\\
		& =  \Big( \prod_{\tv \in \tau \setminus \{\tu, \tv\}} [\deg_{T_n}\phi(\tv)-1]_{\ell(\tv)} \Big) 
		\cdot \Big([\deg_{T_n}\phi(\tu)-1]_{\ell(\tu)} + \ell(\tu) [\deg_{T_n}\phi(\tu)-1]_{\ell(\tu)-1}\Big).\nonumber
	\end{align}
	In the first line, since $\ell(\tv) = 0$, we removed the term coming from $v$ from the product. 
	The quantity above will be weighted by $\bbP\big(\phi(\tu)= u_n\,\big|\, \calE_n \big)= \frac{\deg_{T_n}\phi(\tu)-1+\alpha}{(1+\alpha)n-2}$. 
	In preparation, observe that 
	\begin{align*}
    	&(d-1+\alpha) \cdot [d-1]_{\ell(\tu)} = 
		[d-1]_{\ell(\tu)+1}  + (\ell(\tu)+\alpha)[d-1]_{\ell(\tu)} \qquad \text{ and}\\
    	&(d-1+\alpha) \cdot [d-1]_{\ell(\tu)-1} = 
		[d-1]_{\ell(\tu)}  + (\ell(\tu)-1+\alpha)[d-1]_{\ell(\tu)-1}.
	\end{align*}
	Applying the above with $d =  \deg_{T_n}(\tu)$ to~\eqref{eq:piv}, multiplied by $\bbP\big(\phi(\tu)= u_n\,\big|\, \calE_n \big)$, 
	we find
	\begin{align*}
		&\bbP\big(\phi(\tu)= u_n\,\big|\, \calE_n \big) \cdot \pi(\btau, \phi,T_{n+1}) \\
		& = \tfrac{1}{(1+\alpha)n-2} \cdot \big[
		\pi\big((\btau\setminus \tv)^+, \tilde \phi, T_n\big) + 
		(2\ell(\tu) + \alpha) \cdot \pi(\btau\setminus \tv, \tilde \phi, T_n)
		+\ell(\tu)(\ell(\tu) +\alpha -1) \cdot \pi\big((\btau\setminus \tv)^-, \tilde \phi, T_n\big) \big].
	\end{align*}
	Summing over all embeddings and all values of $\tu,\tv$ we find
	\begin{align}\label{eq:F3}
		&\bbE\Big[ \calF^{(3)}_{\btau}(T_{n+1}) \,\Big|\, \calE_n\Big]\\
		&=
		\tfrac{1}{(1+\alpha)n-2}
		\sum_{\tv \text{ loose leaf}} 
		\calF_{(\btau \setminus \tv)^+}(T_n) +
		(2\ell(\tu) +\alpha) \cdot \calF_{\btau \setminus \tv}(T_n) 
		 + \ell(\tu) (\ell(\tu) +\alpha -1) \cdot \calF_{(\btau \setminus \tv)^-}(T_n), \nonumber	
	\end{align}
	where the sum in the right-hand side is over all loose leaves $\tv$ of $\btau$ and $\tu$ denotes their unique neighbour.
	  
	By summing~\eqref{eq:F12} and~\eqref{eq:F3}, we may obtain a recurrence formula similar to~\eqref{eq:recurrence}, but with one flaw. 
	Indeed, in such a formula the trees $\bsigma$ would potentially be of the type $(\btau \setminus \tv)^+$, hence have same weight as $\btau$. 
	We reduce the contribution of such trees to ones of lower weight via the following lemma.

\begin{lemma}\label{combinatoriallemma}
        Let $\btau=(\tau,\ell)$ be a decorated tree with $w(\btau) \geq 2$, let $\tv$ be a \textit{loose leave} of $\btau$ 
        and $\tu$ its only neighbour in $\tau$. 
        Then, for any tree $T$:
        \begin{align}\label{equalitycombinatoriallemma}
    	    \calF_{(\btau \setminus \tv)^+}(T) 
    	    = \calF_{\btau}(T) + \big(\mathrm{deg}_{\tau}(\tu)-\ell(\tu)-2\big) \cdot \calF_{\btau \setminus \tv}(T).
        \end{align}
    \end{lemma}
    
    \begin{proof}[Proof of Lemma~\ref{combinatoriallemma}]
    We proceed in two steps. Let us first express $\calF_{\btau}(T)$ using $\calF_{\btau \setminus \tv}(T)$. 
    To any $\phi: \tau \hookrightarrow T$, associate its restriction $\tilde\phi : \tau \setminus\tv \to T_n$ to $\tau \setminus\tv$. 
    Conversely, any tree embedding $\tilde\psi : \tau \setminus\tv \hookrightarrow T$ may be extended to some embedding 
    $\psi : \tau \hookrightarrow T$ 
    in as many different ways as there are neighbouring vertices of $\tilde\psi(\tu)$ not reached by $\psi$, 
    that is to say $\mathrm{deg}_{T}(\tilde\psi(\tu))-\mathrm{deg}_{\tau \setminus\tv}(\tu)$ ways. Therefore:
    
    \begin{align}\label{combinatorialequalityonelemma}
        \calF_{\btau}(T) 
        = \sum_{\tilde\psi : \tau \setminus \tv \to T} \big( \mathrm{deg}_{T}(\tilde\psi(\tu))-\mathrm{deg}_{\tau \setminus \tv}(\tu) \big) 
        \cdot \pi(\btau\setminus \tv, \tilde\psi, T).
     \end{align}
    In the above equation, we further use that $\pi(\btau,\psi, T)=\pi(\btau\setminus \tv, \tilde\psi, T)$ given that $\tv$ is a loose leaf. 
    
    Next we express $\calF_{(\btau \setminus \tv)^+}(T)$ in terms of $\calF_{\btau \setminus \tv}(T)$. 
    Recall that $(\btau \setminus \tv)^+$ is obtained from $\btau \setminus \tv$ by increasing the decoration of the vertex $\tu$ by one. Thus:
    
    \begin{align}\label{combinatorialequalitytwolemma}
    	\calF_{(\btau \setminus \tv)^+}(T)  = \sum_{\tilde\psi : \tau \setminus \tv \to T} \big(  \mathrm{deg}_{T}(\tilde\psi(\tu))-1-\ell(\tu) \big) \cdot  \pi(\btau\setminus \tv, \tilde\psi, T).
    \end{align}
    
    Writing $\mathrm{deg}_{T}(\tilde\psi(\tu))-1-\ell(u)=[\mathrm{deg}_{T}(\tilde\psi(\tu))-\mathrm{deg}_{\tau \setminus \tv}(\tu)]+[\mathrm{deg}_{\tau \setminus \tv}(\tu)-1-\ell(\tu)]$ and observing that $\mathrm{deg}_{\tau \setminus \tv}(\tu)=\mathrm{deg}_{\tau}(\tu)-1$, we deduce Lemma~\ref{combinatoriallemma} by subtracting~\eqref{combinatorialequalityonelemma} from~\eqref{combinatorialequalitytwolemma}.
    \end{proof}
	
	By Lemma~\ref{combinatoriallemma}, equation~\eqref{eq:F3} may be re-expressed as
	\begin{align}\label{finaleqproofrecurrenceformula}
		\bbE\big[ \calF^{(3)}_{\btau}(T_{n+1}) \,\big|\, \calE_n\big]\nonumber
		=
		 \tfrac{1}{(1+\alpha)n-2} \sum_{\tv \text{ loose leaf}} &
		\calF_{\btau}(T_n) +
		(\mathrm{deg}_{\tau}(\tu)+\ell(\tu)+\alpha-2) \cdot \calF_{\btau \setminus \tv}(T_n) \\
		 &+ \ell(\tu) (\ell(\tu) +\alpha -1) \cdot \calF_{(\btau \setminus \tv)^-}(T_n).	
	\end{align}
	Since $w(\btau) > 1$, the vertex $\tu$ may never be a loose leaf of $\btau$, 
	hence $\deg_\tau(\tu) \geq 2$ or 
	$\deg_\tau(\tu) =1$ but $\ell(\tu) \geq 1$. 
	In both cases, the multiplicative factor $\mathrm{deg}_{\tau}(\tu)+\ell(\tu)+\alpha-2$ is non-negative. 
	Equations~\eqref{eq:F12} and~\eqref{finaleqproofrecurrenceformula} together yield~\eqref{eq:recurrence}.
\end{proof}

\begin{remark}\label{rem:c>0}
	Following the proof we find that the constants $  \ct(\bsigma,\btau)$ appearing in~\eqref{eq:recurrence} are non-zero 
	only if $\bsigma$ may be obtained from $\btau$ by
	\begin{tight_itemize}
	\item[(i)] decreasing the value of one decoration by $1$ or
	\item[(ii)] removing a loose leaf and conserving all other decorations or
	\item[(iii)] removing a loose leaf and modifying the decoration of its unique neighbour by $-1$.  
	\end{tight_itemize}
	It is direct that all the trees $\bsigma$ are smaller than $\btau$ for $\prec$. 
	Indeed, in most cases we have  $w(\bsigma) < w(\btau)$. 
	However, there are two cases where $w(\bsigma) = w(\btau)$: when the decoration of a leaf is $1$ in $\btau$ and decreases to $0$ in $\bsigma$ (by the procedure (i)) and when a loose leaf is removed from $\btau$ (as in (ii)), with its unique ancestor having decoration $0$, hence becoming a loose leaf of $\bsigma$. 
	In both these cases, $\bsigma \prec \btau$ due to the second condition of Definition~\ref{partialorderdef}.
\end{remark}

\subsection{The first moment of $\calF_{\btau}(T_n)$}\label{sec:first_moment}

We are ready to state the full estimate of the first moment of $\calF_{\btau}(T_n^S)$.

\begin{theorem}\label{thm:first_moment}
    For any $\alpha>0$, any seed tree $S$ of size $k\geq2$ and any decorated tree $\btau$, we have:
    \begin{align*}
	    \mathbb{E}[\calF_{\btau}(T_{n}^{S})] \approx n^{\max\{1,\frac{w(\btau)}{1+\alpha}\}} \cdot (\log n)^{\gamma(\btau)},
    \end{align*}
    where $\gamma(\btau)$ is a nonnegative exponent equal to zero when ${w(\btau)<1+\alpha}$ and otherwise recursively defined by:
    \begin{align*}%\label{logexponent}
    \gamma(\btau) = \sup_{\bsigma \prec \btau, \ c(\bsigma,\btau)>0, \ w(\bsigma)=w(\btau)}  (\gamma(\bsigma)+1) \quad \text{if} \ {w(\btau)>1+\alpha},
    \end{align*}
    or by \emph{(critical case)}:
    \begin{align*}%\label{criticallogexponent}
    \gamma(\btau) = \gamma_{c}(\btau) 
    := \max \bigg\{1 \ , \ \sup_{\bsigma \prec \btau, \ c(\bsigma,\btau)>0, \ w(\bsigma)=w(\btau)}  (\gamma_{c}(\bsigma)+1)\bigg\} 
    \quad \text{if} \ {w(\btau)=1+\alpha}.
    \end{align*}
    with the convention $\sup \emptyset = 0$.
\end{theorem}

\begin{remark}\label{rem:no_log}
	Let $\btau$ be a decorated tree with $\ell(u) \geq 2$ for all its leaves. In particular, it contains no loose leaves and $w(\btau)=|\ell|=\sum_{u \in \tau} \ell(u)$. 
	Then Remark~\ref{rem:c>0} indicates that $\ct(\bsigma,\btau)>0$ only for decorated trees $\bsigma$ obtained by 
	lowering the decoration of some vertex $u \in \btau$ by $1$. 
	Moreover, no such tree $\bsigma$ has any loose leaf either. Assuming in addition that $w(\btau)=|\ell|>1+\alpha$, we have $w(\bsigma) < w(\btau)$ and $\gamma(\btau)=0$.
\end{remark}

To prove Theorem~\ref{thm:first_moment}, we proceed by induction on the set of decorated trees (for the partial order $\preceq$) 
and use the recurrence formula~\eqref{eq:recurrence}.

\begin{proof}[Proof of Theorem~\ref{thm:first_moment}]
We first note that for the three decorated trees $\btau$ with $w(\btau)=1$ (thus $w(\btau) < 1+\alpha$ for any $\alpha>0$) the first moment estimate is satisfied as indicated by the explicit formula of Proposition~\ref{prop:recurrence}. 

Consider now a decorated tree $\btau$ with $w(\btau)\geq 2$ and suppose by induction that Theorem~\ref{thm:first_moment} is valid for all decorated trees $\bsigma \prec \btau$.  Let us define for any $n\geq k$ the quantity $\omega_{n+1}^{(\btau)}$ equal to:
\begin{align}\label{polynomialgrowthfactor}
	\omega_{n+1}^{(\btau)}:= \prod_{\ell = k}^{n} \Big(1+\frac{w(\btau)}{(1+\alpha)\ell-2}\Big)^{-1}
	\approx n^{-\frac{w(\btau)}{1+\alpha}},
\end{align}
where the latter equivalent is obtained by a standard computation. 
Then, by multiplying~\eqref{eq:recurrence} by this factor and taking the expectation, we get:
\begin{align*}%\label{recurrenceformulascaled}
	\mathbb{E}[\omega_{n+1}^{(\btau)} \cdot \calF_{\btau}(T_{n+1})] 
	= \mathbb{E}[\omega_{n}^{(\btau)} \cdot \calF_{\btau}(T_{n})] 
	+ \sum_{\bsigma \prec \btau} c(\bsigma,\btau) \cdot \frac{\omega_{n+1}^{(\btau)}}{(1+\alpha)n-2}  \cdot\mathbb{E}[\calF_{\bsigma}(T_{n})].
\end{align*}
By iterating the above over $n$, we find
\begin{align}\label{recurrenceformulagoodform}
    \mathbb{E}[\omega_{n+1}^{(\btau)} \cdot \calF_{\btau}(T_{n+1})] 
    = \calF_{\btau}(S) + \sum_{\bsigma \prec \btau} c(\bsigma,\btau) \sum_{\ell=k}^{n} \frac{\omega_{\ell+1}^{(\btau)} }{(1+\alpha)\ell-2}\cdot \mathbb{E}[\calF_{\bsigma}(T_{\ell})].
\end{align}
Thus, the asymptotic behaviour of $\mathbb{E}[\calF_{\btau}(T_{n+1})]$ can be derived from that of $\mathbb{E}[\calF_{\bsigma}(T_{\ell})]$ for $\bsigma \prec \btau$ and $\ell \leq n$.
Define the following variables 
\begin{align*}%\label{sumforsigmaandtauinrecurrenceformula}
	\mathcal{S}_{n}(\bsigma,\btau)
	:= \sum_{\ell=k}^{n} \frac{\omega_{\ell+1}^{(\btau)} }{(1+\alpha)\ell-2}\cdot \mathbb{E}[\calF_{\bsigma}(T_{\ell})].
\end{align*}
The growth rate in $n$ of $\mathcal{S}_{n}(\bsigma,\btau)$ depends on $w(\bsigma)$ through $\mathbb{E}[\calF_{\bsigma}(T_{\ell})]$. 
We distinguish three cases according to the value of $w(\btau)$.
\\
\\
\noindent\boxed{\textbf{1st case: $w(\btau)<1+\alpha$}} \\
\\
In this situation, we necessarily have $w(\bsigma) < 1+\alpha$ for any $\bsigma \prec \btau$. Thus, by  the induction hypothesis:
\begin{align*}%\label{inductionhypothesisonfirstmomentsfirstcase}
	\mathbb{E}[\calF_{\bsigma}(T_{\ell})] \approx \ell.
\end{align*}
The above together with~\eqref{polynomialgrowthfactor} imply that %and~\eqref{growthfactorinverseinn} imply that:
\begin{align*}%\label{growthincrementsumforsigmaandtauinrecurrenceformulafirstcase}
	\frac{\omega_{\ell+1}^{(\btau)} }{(1+\alpha)\ell-2} \cdot  \mathbb{E}[\calF_{\bsigma}(T_{\ell})] \approx \ell^{-\frac{w(\btau)}{1+\alpha}}.
\end{align*}
Since $w(\btau)<1+\alpha$, the sum over $\ell$ of the above - which constitutes $\mathcal{S}_{n}(\bsigma,\btau)$ - diverges at rate:
\begin{align*}%\label{ratesumsigmataufirstcase}
	\mathcal{S}_{n}(\bsigma,\btau) 
	= \sum_{\ell=k}^{n} \frac{1}{(1+\alpha)\ell-2} \cdot \omega_{\ell+1}^{(\btau)} 
	\cdot \mathbb{E}[\calF_{\bsigma}(T_{\ell})] \approx n^{1-\frac{w(\btau)}{1+\alpha}}.
\end{align*}
We then sum over every $\bsigma \prec \btau$ to get an asymptotic estimate for the quantity on the right of~\eqref{recurrenceformulagoodform}:
\begin{align}\label{finalestimatefirstcasefirstmoment}
	\mathbb{E}[\omega_{n+1}^{(\btau)} \cdot \calF_{\btau}(T_{n+1})] 
	= \calF_{\btau}(S) + \sum_{\bsigma \prec \btau} c(\bsigma,\btau) \sum_{\ell=k}^{n} 
	\frac{\omega_{\ell+1}^{(\btau)} }{(1+\alpha)\ell-2} \cdot \mathbb{E}[\calF_{\bsigma}(T_{\ell})] 
	\approx n^{1-\frac{w(\btau)}{1+\alpha}},
\end{align}
since there exists at least one $\bsigma$ with $c(\bsigma,\btau)> 0$.
Finally, dividing~\eqref{finalestimatefirstcasefirstmoment} by $\omega_{n+1}^{(\btau)}$ and using~\eqref{polynomialgrowthfactor}, we obtain the expected estimate:
\begin{align*}
\mathbb{E}[\calF_{\btau}(T_{n+1})] \approx n.
\end{align*} 
\\
\noindent\boxed{\textbf{2nd case: $w(\btau)=1+\alpha$}} \\
\\
Now, when $\bsigma \prec \btau$, we can either have $w(\bsigma) < 1+\alpha$ or $w(\bsigma) = 1+\alpha$. In the former situation,
we have 
\begin{align*}%\label{growthincrementsumforsigmaandtauinrecurrenceformulasecondcaseunder}
	\frac{1}{(1+\alpha)n-2} \cdot \omega_{n+1}^{(\btau)} \cdot \mathbb{E}[\calF_{\bsigma}(T_{n})] \approx n^{-1},
\end{align*}
hence, 
\begin{align*}%\label{ratesumsigmataufirstcaseunder}
	\mathcal{S}_{n}(\bsigma,\btau) 
	= \sum_{\ell=k}^{n} \frac{1}{(1+\alpha)\ell-2} \cdot \omega_{\ell+1}^{(\btau)} \cdot \mathbb{E}[\calF_{\bsigma}(T_{\ell})] 
	\approx \log(n).
\end{align*}
In the latter situation, by the induction hypothesis,
\begin{align*}%\label{inductionhypothesisonfirstmomentsfirstcase}
	\mathbb{E}[\calF_{\bsigma}(T_{n+1})] \approx n\cdot (\log{n})^{\gamma_{c}(\bsigma)}.
\end{align*}
Using~\eqref{polynomialgrowthfactor}, we find, 
\begin{align*}%\label{growthincrementsumforsigmaandtauinrecurrenceformulasecondcase}
	\frac{1}{(1+\alpha)n-2} \cdot \omega_{n+1}^{(\btau)} \cdot \mathbb{E}[\calF_{\bsigma}(T_{n})] 
	\approx n^{-\frac{w(\btau)}{1+\alpha}} \cdot (\log{n})^{\gamma_{c}(\bsigma)} = n^{-1} \cdot (\log{n})^{\gamma_{c}(\bsigma)}.
\end{align*}
As a consequence, $\mathcal{S}_{n}(\bsigma,\btau)$ diverges at rate:
\begin{align*}%\label{ratesumsigmatausecondcase}
	\mathcal{S}_{n}(\bsigma,\btau) 
	%= \sum_{\ell=k}^{n} \frac{1}{(1+\alpha)\ell-2} \cdot \omega_{\ell+1}^{(\btau)} \cdot \mathbb{E}[\calF_{\bsigma}(T_{\ell})] 
	\approx \sum_{\ell=k}^{n} {\ell}^{-1} \cdot   (\log{\ell})^{\gamma_{c}(\bsigma)}
		\approx (\log{n})^{\gamma_{c}(\bsigma)+1}.
\end{align*}
Thus, only the terms $\bsigma$ with maximal weight and $\ct(\bsigma,\btau)>0$ contribute to~\eqref{recurrenceformulagoodform} significantly:
\begin{align*}%\label{finalestimatesecondcasefirstmoment}
	\mathbb{E}[\omega_{n+1}^{(\btau)} \cdot \calF_{\btau}(T_{n+1})] 
%	= \calF_{\btau}(S) + \sum_{\bsigma \prec \btau} c(\bsigma,\btau) \sum_{\ell=k}^{n} 
%	\frac{1}{(1+\alpha)\ell-2} \cdot \omega_{\ell+1}^{(\btau)} \cdot \mathbb{E}[\calF_{\bsigma}(T_{\ell})] 
	\approx
	\sum_{\substack{\bsigma \prec \btau \\w(\bsigma) < 1+\alpha}}\!\!\! c(\bsigma,\btau) \cdot \log{n} +
	\sum_{\substack{\bsigma \prec \btau \\w(\bsigma) = 1+\alpha}}\!\!\! c(\bsigma,\btau) \cdot (\log{n})^{\gamma_{c}(\bsigma)+1}
	\approx (\log{n})^{\gamma_{c}(\btau)}.
\end{align*}
where $\gamma_{c}(\btau)$ is defined as in Theorem~\ref{thm:first_moment}. 
Dividing the last equation by $\omega_{n+1}^{(\btau)}$ leads to the result. \\
\\
\noindent\boxed{\textbf{3rd case: $w(\btau)>1+\alpha$}} \\
\\
In this case, the trees $\bsigma \prec \btau$ can satisfy either $w(\bsigma)<1+\alpha$, $w(\bsigma)=1+\alpha$, $1+\alpha < w(\bsigma) < w(\btau)$, or $1+\alpha < w(\bsigma)=w(\btau)$. 
In the two first situations, by the induction hypothesis, there exists $\delta>0$ such that:
\begin{align*}%\label{growthincrementsumforsigmaandtauinrecurrenceformulathirdcaseunderorequal}
	\frac{1}{(1+\alpha)\ell-2} \cdot \omega_{\ell+1}^{(\btau)} \cdot \mathbb{E}[\calF_{\bsigma}(T_{\ell})] 
	\ll \ell^{-\frac{w(\btau)}{1+\alpha}} \cdot (\log \ell)^{\delta}.
\end{align*}
Since $w(\btau)>1+\alpha$, the sum over $\ell$ of the above converges and $\mathcal{S}_{n}(\bsigma,\btau) \approx 1$.

When $1+\alpha < w(\bsigma) < w(\btau)$, the induction hypothesis implies:
\begin{align*}%\label{growthincrementsumforsigmaandtauinrecurrenceformulathirdcasgreaterunder}
    \frac{1}{(1+\alpha)\ell-2} \cdot \omega_{\ell+1}^{(\btau)} \cdot \mathbb{E}[\calF_{\bsigma}(T_{\ell})] 
    \approx \ell^{\frac{w(\bsigma)-w(\btau)}{1+\alpha}-1} \cdot (\log \ell)^{\gamma(\bsigma)}.
\end{align*}
%where $\gamma(\bsigma)$ is defined as in~\eqref{logexponent}. 
Since $w(\bsigma)<w(\btau)$, the sum over $\ell$ of the above converges again, and $\mathcal{S}_{n}(\bsigma,\btau) \approx 1$.

Finally, when $w(\bsigma)=w(\btau)$, the induction hypothesis gives us:
\begin{align*}%\label{growthincrementsumforsigmaandtauinrecurrenceformulathirdcasgreaterequal}
    \frac{1}{(1+\alpha)\ell-2} \cdot \omega_{\ell+1}^{(\btau)} \cdot \mathbb{E}[\calF_{\bsigma}(T_{\ell})] 
    \approx \ell^{-1} \cdot (\log \ell)^{\gamma(\bsigma)}.
\end{align*}
%where $\gamma(\bsigma)$ is defined as in~\eqref{logexponent}. 
Consequently, by a direct computation, the divergence rate of $\mathcal{S}_{n}(\bsigma,\btau)$ may be shown to be:
\begin{align*}%\label{ratesumsigmatausecondcasegreaterequal}
    \mathcal{S}_{n}(\bsigma,\btau) 
    %= \sum_{\ell=k}^{n} \frac{1}{(1+\alpha)\ell-2} \cdot \omega_{\ell+1}^{(\btau)} \cdot \mathbb{E}[\calF_{\bsigma}(T_{\ell})] 
    \approx \sum_{\ell=k}^{n} \ell^{-1} \cdot (\log \ell)^{\gamma(\bsigma)}
    \approx (\log{n})^{\gamma(\bsigma)+1}.
\end{align*}
In conclusion, by considering the asymptotic of $\mathcal{S}_{n}(\bsigma,\btau)$ for all $\bsigma \prec \btau$ with $\ct(\bsigma,\btau)>0$
according to the above, we obtain
\begin{align*}
	\mathbb{E}[\omega_{n+1}^{(\btau)} \cdot \calF_{\btau}(T_{n+1})] 
	\approx
	\sum_{\substack{\bsigma \prec \btau \\ w(\bsigma) < w(\btau)}}\!\!\! c(\bsigma,\btau)  +
	\sum_{\substack{\bsigma \prec \btau \\ w(\bsigma) = w(\btau)}}\!\!\! c(\bsigma,\btau) \cdot (\log{n})^{\gamma(\bsigma)+1}
	\approx (\log{n})^{\gamma(\btau)},
\end{align*}
where the last equivalent is due to how $\gamma(\btau)$ is defined in Theorem~\ref{thm:first_moment}. 
Divide by $\omega_{n+1}^{(\btau)}$ to obtain the expected result.
\end{proof}

\subsection{The second moment of $\calF_{\btau}(T_n)$}\label{secondmomentestimatetheoremsection}

We are now ready to prove the second moment estimate on $\mathcal{F}_{\btau}(T_n)$ of Theorem~\ref{thm:second_moment}.
We will build on the analogous result on the first moment obtained in the previous section as well as on its proof. 
First, remark that the square of the observables may be written as:
\begin{align}\label{squareofobservables}
    \calF_{\btau}(T_n)^2
    = \sum_{\phi_1,\phi_2 : \tau \hookrightarrow T_n} \, \pi(\btau,\phi_1, T_n) \cdot \pi(\btau,\phi_2, T_n),
\end{align}
where the sum is this time over all pairs $\phi_1,\phi_2$ of graph embeddings of $\tau$ in $T_n$. 
We will decompose the sum appearing in~\eqref{squareofobservables} in two parts, according to whether the embeddings $\phi_1(\tau)$ and $\phi_2(\tau)$ overlap or not. The same is done in \cite{curien2015scaling}. Call $\calF_{\btau,\btau}(T_n)$ the first resulting quantity and $\calF_{\btau+\btau}(T_n)$ the second one:
\begin{align*}
     \calF_{\btau,\btau}(T_n)
    = \!\!\!\!\!\sum_{\substack{\phi_1,\phi_2 : \tau \to T_n\\ \phi_1(\tau)\cap \phi_2(\tau) = \emptyset}} \!\!\!\!\!
    \pi(\btau,\phi_1, T_n)  \pi(\btau,\phi_2, T_n) ,
    \quad 
         \calF_{\btau+\btau}(T_n)
    = \!\!\!\!\!\sum_{\substack{\phi_1,\phi_2 : \tau \to T_n\\ \phi_1(\tau)\cap \phi_2(\tau) \neq \emptyset}} \!\!\!\!\!
    \pi(\btau,\phi_1, T_n)  \pi(\btau,\phi_2, T_n).
%
 %   \calF_{\btau,\btau}(T_n) = \calF_{\btau}(T_n)^2 - \calF_{\btau+\btau}(T_n).
%    = \sum_{\substack{\phi_1,\phi_2 : \tau \to T_n\\ \phi_1(\tau)\cap \phi_2(\tau) = \emptyset}} 
%    \pi(\btau,\phi_1, T_n) \cdot \pi(\btau,\phi_2, T_n),
\end{align*}

The first moments of the two quantities above are bounded separately. First we estimate $\bbE[\calF_{\btau+\btau}(T_n)]$, 
which turns out to be the the easiest of the two. Indeed, it may be expressed  as a sum of first moments 
of observables for some decorated trees derived from $\btau$. These are computed using Theorem~\ref{thm:first_moment}. 
To deal with $\mathbb{E}[\calF_{\btau,\btau}(T_n)]$ we will prove a recurrence inequality on such quantities, similar to~\eqref{eq:recurrence} and using the same techniques.

\begin{subsubsection}{An estimate on $\mathbb{E}[\calF_{\btau+\btau}(T_n)]$}\label{sectionsecmompartone}

The goal of this section is to show the following:
\begin{proposition}\label{prop:tau+tau}
	Let $\btau$ be a decorated tree with $\ell(u) \geq 2$ for any $u \in \tau$
    and such that $|\ell| = \sum_{u \in \tau} \ell(u) > 1 + \alpha$. 
   	Then, for any seed $S$,
	\begin{align}\label{eq:tau+tau}
	    \mathbb{E}[\mathcal{F}_{\btau+\btau}(T_{n}^S)] = \calO\big( n^{\frac{2 |\ell|}{1+\alpha}}\big).
    \end{align}
\end{proposition}

The proposition is based on the following lemma, which we prove below.

\begin{lemma}\label{lem:tau+tau2}
	Let $\btau = (\tau,\ell)$ be a decorated tree. 
	There exists a finite set $\calU(\btau)$ of decorated trees $\bsigma$ with $w(\bsigma) \leq 2w(\btau)$  
	and positive constants $C(\btau,\bsigma)$ for $\bsigma \in \calU(\btau)$ (see the proof for an explicit description) such that, 
	for any tree $T$, 
	\begin{align}\label{eq:tau+tau2}
    	\calF_{\btau+\btau}(T) 
    	= \sum_{\bsigma \in \calU(\btau)} C(\btau,\bsigma) \cdot \calF_{\bsigma}(T).
	\end{align}
\end{lemma}

In the planar setting the above is very intuitive; we sketch a proof below.
If $T$ is plane, then $\calF_{\btau+\btau}(T)$ is the number of decorated embeddings $(\phi_1,\phi_2)$ 
of two copies $\tau_1$ and $\tau_2$ of $\tau$ in $T$, which overlap. 

Let us first forget about the decorations and focus on graph embeddings. 
The union of the images of $\tau_1$ and $\tau_2$ via such embeddings is a tree $\sigma$; 
one may see $\sigma$ as a merger of $\tau_1$ and $\tau_2$. 
Thus, the pairs of overlapping embeddings of $\tau_1$ and $\tau_2$ in $T$ are in bijection with the embeddings of $\sigma$ in $T$ where $\sigma$ ranges over all possible mergers of $\tau_1$ and $\tau_2$.

Now consider overlapping decorated embeddings of $\tau_1$ and $\tau_2$ in $T$. 
Then, each corner of $T$ may have no arrow pointing to it, an arrow of $\tau_1$, 
an arrow of $\tau_2$ or one arrow of $\tau_1$ and one of $\tau_2$ pointing to it.
In the first three cases, the arrows pointing to the corner will be considered as arrows of $\sigma$; 
in the last case, the arrow of $\tau_1$ and that of $\tau_2$ merge into a single arrow of $\sigma$. 
Thus, any such pair of decorated embeddings corresponds to a decorated embedding of some decorated tree  $\bsigma = (\sigma,m)$ obtained as a merger of $\btau_1$ and $\btau_2$. In particular $|m| \leq |\ell_1| + |\ell_2| = 2|\ell|$. 
The constants $C(\btau,\bsigma)$ are combinatorial factors that account for the different ways to merge arrows of $\btau_1$ and $\btau_2$. 

The actual proof given below avoids the use of the planar model and is more algebraic.

\begin{proof}[Proof of Lemma \ref{lem:tau+tau2}]
Fix $\btau$ and $T$ as in the lemma. 
Consider two embeddings $\Phi_1,\Phi_2$ of $\tau$ in $T$ whose images intersect. 
Since $T$ is a tree, $\Phi_1(\tau) \cup \Phi_2(\tau)$ is itself a subtree of $T$.
Moreover, $\Phi_1(\tau) \cap \Phi_2(\tau)$ is the image via $\Phi_1$ and $\Phi_2$, respectively, of two isomorphic subtrees $\sigma_1$ and $\sigma_2$ of $\tau$. 

For two isomorphic subtrees $\sigma_1$ and $\sigma_2$ of $\tau$, define the $(\sigma_1,\sigma_2)$-merger of two copies of $\tau$
as the tree obtained by ``gluing'' two copies of $\tau$ along $\sigma_1$ and $\sigma_2$ respectively.
Write $\calM(\sigma_1,\sigma_2)$ for this tree. 
To identify the two copies of $\tau$ merged to obtain $\calM(\sigma_1,\sigma_2)$, call them $\btau_1 = (\tau_1,\ell_1)$ and 
$\btau_2 = (\tau_2,\ell_2)$.

Each vertex of $\calM(\sigma_1,\sigma_2)$ is clearly identified to either one vertex in $\tau_1 \setminus \sigma_1$, 
a vertex in $\tau_2 \setminus \sigma_2$ or to a vertex in $\sigma_1$ and simultaneously to one in $\sigma_2$. 
For a vertex $u \in \calM(\sigma_1,\sigma_2)$ write $\ell_1(u)$ for its decoration in $\btau_1$, if it is identified to a vertex of $\tau_1$, 
otherwise set $\ell_1(u) = 0$. Define $\ell_2(u)$ for $u \in \calM(\sigma_1,\sigma_2)$ in the same way. 

Then, the pairs of embeddings $\Phi_1,\Phi_2$ of $\tau$ in $T$ with $\Phi_1(\tau) \cap \Phi_2(\tau) = \Phi_1(\sigma_1) = \Phi_2(\sigma_2)$
are in bijection with the embeddings of $\calM(\sigma_1,\sigma_2)$ in $T$. 
It follows that 
\begin{align}\label{eq:tau+tau3}
	\calF_{\btau+\btau}(T) =
	\sum_{(\sigma_1,\sigma_2)} \, \sum_{\phi : \calM(\sigma_1,\sigma_2) \hookrightarrow T}\, \prod_{u \in \calM(\sigma_1,\sigma_2)}
	[\deg_{T}(\phi(u))-1]_{\ell_1(u)} \cdot [\deg_{T}(\phi(u))-1]_{\ell_2(u)},
\end{align}
where the first sum is over all pairs of isomorphic subtrees $(\sigma_1, \sigma_2)$ of $\tau$. 
To reduce the above to a formula of the type~\eqref{eq:observable}, we use the following combinatorial identity. 

\begin{fact}\label{lem:[n]}
	Fix non-negative integers $n, \ell_1,\ell_2$. Then 
	\begin{align*}
		[n]_{\ell_1}\cdot [n]_{\ell_2}
%		 = \sum_{j \geq 0 } [n]_{\ell_1 + \ell_2 - j}
%		 \frac{1}{(\ell_1 + \ell_2 - j)!}
%		 [\ell_1 + \ell_2 - j]_{\ell_1}\cdot  \ell_2 ! 
		 = \sum_{j = 0 }^{\min\{\ell_1,\ell_2\}}\, [n]_{\ell_1 + \ell_2 - j} 	\cdot \frac{\ell_1! \cdot \ell_2!}{(\ell_1 - j)!\,(\ell_2 - j)!\,j!}.
		 \end{align*}
\end{fact}

\begin{proof}
	First observe that $[n]_\ell = \binom{n}{\ell} \cdot \ell!$.
	Now let us inspect the quantity $\binom{n}{\ell_1}	\binom{n}{\ell_2} $ 
	which is the number of pairs of subsets $A_1,A_2$ of $\{1,\dots, n\}$ with $\ell_1$ and $\ell_2$ elements, respectively. 
	These may be counted as follows. 
	First establish the set $A_1 \cup A_2$ which can have $\ell_1 + \ell_2 - j$ elements with $j \in \{0,\dots,\min\{\ell_1,\ell_2\} \}$.  
	Once $A_1 \cup A_2$ is chosen, split it into $A_1 \setminus A_2$, $A_2 \setminus A_1$ and $A_1 \cap A_2$. 
	This forms a partition of $A_1 \cup A_2$ into three sets of cardinality $\ell_1 - j$, $\ell_2 - j$ and $j$, respectively. 
	Thus
	\begin{align*}
		\binom{n}{\ell_1} \cdot	\binom{n}{\ell_2} 
		=\sum_{j = 0 }^{\min\{\ell_1,\ell_2\}} \binom{n}{\ell_1 + \ell_2 - j} \frac{(\ell_1 + \ell_2 - j)!}{(\ell_1 - j)!\,(\ell_2 - j)!\,j!}.
	\end{align*}
	Multiply by $\ell_1! \cdot \ell_2!$ to find the desired result. 
\end{proof}

Let us return to the proof of Lemma \ref{lem:tau+tau2}.
A valid decoration $m$ for a tree in $\calM(\sigma_1,\sigma_2)$ is one such that, for all $u \in \calM(\sigma_1,\sigma_2)$, 
\begin{align}\label{eq:valid_decoration}
	\max\{\ell_1(u), \ell_2(u)\} \leq  m(u) \leq \ell_1(u) + \ell_2(u). 
\end{align}
Observe that, if $u$ is not in the images of $\sigma_1$ and $\sigma_2$, then $m(u)$ is entirely determined by the above. 
For vertices than are in the overlap of $\tau_1$ and $\tau_2$, $m(u)$ may take one of several values.

Applying Fact~\ref{lem:[n]} to the summand in~\eqref{eq:tau+tau3}, we find
\begin{align*}
	 &\prod_{u \in \calM(\sigma_1,\sigma_2)}	[\deg_{T}(\phi(u))-1]_{\ell_1(u)} \cdot [\deg_{T}(\phi(u))-1]_{\ell_2(u)}\\
	 &= 	\prod_{u \in \calM(\sigma_1,\sigma_2)}\sum_{j = 0 }^{\min\{\ell_1(u),\ell_2(u)\}}\, [\deg_{T}(\phi(u))-1]_{\ell_1(u) + \ell_2(u) - j} 	
	  \cdot \frac{\ell_1(u)! \cdot \ell_2(u)!}{(\ell_1(u) - j)!\,(\ell_2(u) - j)!\,j!}	\\
	 &= \sum_{m}
	\prod_{u \in \calM(\sigma_1,\sigma_2)}
	[\deg_{T}(\phi(u))-1]_{m(u)}\cdot 
	\frac{\ell_1(u)! \cdot \ell_2(u)!}{(m(u) - \ell_2(u))!(m(u) - \ell_1(u))!(\ell_1(u) + \ell_2(u) - m(u))!
} \\
&= \sum_{m}
	C(\btau, (\calM(\sigma_1,\sigma_2), m))
	\prod_{u \in \calM(\sigma_1,\sigma_2)}
	[\deg_{T}(\phi(u))-1]_{m(u)},
\end{align*}
where the sum in the last two terms is over all valid decorations $m$ of $\calM(\sigma_1,\sigma_2)$
and 
\begin{align*}
	C(\btau,(\calM(\sigma_1,\sigma_2), m)) 
	:= \prod_{u \in \calM(\sigma_1,\sigma_2)} \frac{\ell_1(u)! \cdot \ell_2(u)!}{(m(u) - \ell_2(u))!(m(u) - \ell_1(u))!(\ell_1(u) + \ell_2(u) - m(u))!}.
\end{align*}
Write $\boldsymbol{\calM}$ for the tree $\calM(\sigma_1,\sigma_2)$ with decoration $m$. 
Then 
$$\sum_{\phi: \calM(\sigma_1,\sigma_2) \hookrightarrow T} \prod_{u \in \calM(\sigma_1,\sigma_2)}[\deg_{T}(\phi(u))-1]_{m(u)} = \calF_{\boldsymbol{\calM}}(T).$$
Inserting this into~\eqref{eq:tau+tau3}, we find 
\begin{align*}
	\calF_{\btau+\btau}(T)
	= \sum_{\boldsymbol{\calM}} C(\btau,\boldsymbol{\calM}) \calF_{\boldsymbol{\calM}}(T), 
\end{align*}
where the sum is over all trees of the form $\calM(\sigma_1,\sigma_2)$ with valid decorations $m$. These form the set $\calU(\btau)$;
it is immediate that they have weight at most  $2 w(\btau)$.
\end{proof}

We are finally ready to prove Proposition~\ref{prop:tau+tau}
\begin{proof}[Proof of Proposition~\ref{prop:tau+tau}]
	Fix a decorated tree as in the statement. 
	%Then the trees $\bsigma \in \calU(\btau)$ have weight at most equal to $2 w(\btau)$. 
	By Theorem~\ref{thm:first_moment}, for any tree $\bsigma \in \calU(\btau)$ of weight strictly smaller than $2w(\btau)$, 
	\begin{align*}
		\bbE[\calF_{\bsigma}(T_n^S)] \ll n^{\frac{2w(\btau)}{1+\alpha}}.
	\end{align*}
	Moreover, if $\bsigma  = (\sigma,m) \in \calU(\btau)$ is such that $w(\bsigma) = 2w(\btau)$ 
	then its decorations satisfy 
	$m(u) = \ell_1(u) + \ell_2(u) \geq 2$ for all $u \in \sigma$ (see~\eqref{eq:valid_decoration} for how $m$ is defined). 
	As explained in Remark~\ref{rem:no_log}, for any such tree 
	\begin{align*}
		\bbE[\calF_{\bsigma}(T_n^S)] \approx n^{\frac{2w(\btau)}{1+\alpha}}.
	\end{align*}
	Using~\eqref{eq:tau+tau2} and observing that $\btau$ has no loose leaf, we conclude that 
	\begin{align*}
		\bbE[\calF_{\btau+\btau}(T_n^S)] = \calO\big(n^{\frac{2w(\btau)}{1+\alpha}}\big)= \calO\big(n^{\frac{2|\ell|}{1+\alpha}}\big).
	\end{align*}
	
\end{proof}

\end{subsubsection}

\begin{subsubsection}{An estimate on $\mathbb{E}[\calF_{\btau,\btau}(T_n)]$}\label{sectionsecmomparttwo}

We start this section by defining a wider class of observables that will be involved in a recurrence relation which will eventually allow us to estimate  $\mathbb{E}[\calF_{\btau,\btau}(T_n)]$. 
Let $\btau,\bsigma$ be two decorated trees and $T$ a (bigger) tree. 
Then we denote by $\calF_{\btau,\bsigma}(T)$ the following integer-valued observable:
\begin{align}\label{defobservablesecmomembeddingsnooverlap}
    \calF_{\btau,\bsigma}(T) := \underset{\phi_1(\tau)\cap\phi_2(\sigma)
    =\emptyset}{\sum_{\phi_1,\phi_2}} \, \pi(\btau,\phi_1, T) \cdot \pi(\bsigma,\phi_2, T),
\end{align}
where the sum is over all graph embeddings $\phi_1$, resp. $\phi_2$, of $\tau$, resp. $\sigma$, in $T$ with no overlap in their image. 
The quantity of interest to us is that with $\bsigma=\btau$ and $T=T_n$. 
This section is concerned with proving the following bound. 

\begin{proposition}\label{propupperboundsecmomembeddingsnooverlap}
    Let $T_n$ be an $\alpha$-PA tree. Then it holds that: 
    \begin{align}\label{asymptoticupperboundsecmomembeddingsnooverlap}
    	\mathbb{E}[\calF_{\btau,\bsigma}(T_n)] 
		= \mathcal{O}\big( \mathbb{E}[\calF_{\btau}(T_n)] \cdot \mathbb{E}[\calF_{\bsigma}(T_n)]\big).
    \end{align}
\end{proposition}

Proposition~\ref{propupperboundsecmomembeddingsnooverlap} implies directly the bound necessary for the proof of Theorem~\ref{thm:second_moment}. 

\begin{corollary}\label{cor:tautau}
	Let $\btau$ be a decorated tree with $\ell(u) \geq 2$ for any $u \in \tau$
    and such that $|\ell| > 1 + \alpha$. 
   	Then, for any seed $S$,
	\begin{align}\label{eq:tautau}
	    \mathbb{E}[\mathcal{F}_{\btau,\btau}(T_{n}^S)] = \calO\big( n^{\frac{2 |\ell|}{1+\alpha}}\big).
    \end{align}
\end{corollary}

\begin{proof}[Proof of Corollary~\ref{cor:tautau}]
	Apply Proposition~\ref{propupperboundsecmomembeddingsnooverlap} with $\bsigma=\btau$
	and observe that, due to Remark~\ref{rem:no_log} and the conditions on $\btau$, we have $\mathbb{E}[\calF_{\btau}(T_n)] \approx n^{\frac{|\ell|}{1+\alpha}}$. 
\end{proof}

Proposition~\ref{propupperboundsecmomembeddingsnooverlap} is obtained through a recurrence relation, 
similarly to how Theorem~\ref{thm:first_moment} follows from Proposition~\ref{prop:recurrence}. 
Since we only need an upper bound, we state only a (simpler) recursive inequality. 

\begin{proposition}\label{propinequalityformularecpropsecmomembeddingsnooverlap}
	There exists a family of nonnegative real numbers 
	$\{c^{*}(\boldsymbol{\theta'},\boldsymbol{\theta}) \ : \ \boldsymbol{\theta'} \prec 	\boldsymbol{\theta}\}$ 
	such that for any two decorated trees $\btau,\bsigma \neq \small{\textcircled{\tiny{0}}}$:
	\begin{align}\label{inequalityformularecpropsecmomembeddingsnooverlap}
    	\bbE[\calF_{\btau,\bsigma}(T_{n+1})) \vert \mathcal{F}_n] 
    	&\leq \Big(1+\frac{w(\btau)+w(\bsigma)}{(1+\alpha)\cdot n-2}\Big)\cdot \calF_{\btau,\bsigma}(T_{n}) \nonumber \\ 
    	&+ \frac{1}{(1+\alpha)\cdot n-2} \Big[
    	\sum_{\boldsymbol{t} \prec \btau} c^{*}(\boldsymbol{t},\btau) \cdot \calF_{\boldsymbol{t},\bsigma}(T_{n}) 
    	+ \sum_{\boldsymbol{s} \prec \bsigma} c^{*}(\boldsymbol{s},\bsigma) \cdot \calF_{\btau,\boldsymbol{s}}(T_{n})\Big].
	\end{align}
\end{proposition}

In the rest of the section, we show how Proposition~\ref{propinequalityformularecpropsecmomembeddingsnooverlap}, 
implies Proposition~\ref{propupperboundsecmomembeddingsnooverlap}, 
then prove Proposition~\ref{propinequalityformularecpropsecmomembeddingsnooverlap}. 
Both proofs follow similar arguments to those in Sections~\ref{sec:first_moment} and~\ref{sec:reucrrence}, respectively.
%The proof of the latter is postponed at the end of this section given it is secondary compared to the main prupose of this paper. Furthermore, the proof is obtained bringing only slight modifications to that of Proposition~\ref{prop:recurrence} in section~\ref{sec:reucrrence}.

\begin{proof}[Proof of Proposition~\ref{propupperboundsecmomembeddingsnooverlap}]
We are going to proceed by induction on pairs $(\btau,\bsigma)$ of decorated trees, for the partial order induced by $\prec$
on such pairs (precisely $(\boldsymbol{t},\boldsymbol{s}) \prec (\btau,\bsigma)$ if either $\boldsymbol{t} \preceq \btau$ and $\boldsymbol{s} \prec \bsigma$ or $\boldsymbol{t} \prec \btau$ and $\boldsymbol{s} \preceq \bsigma$).

\paragraph*{Base case}
We show~\eqref{asymptoticupperboundsecmomembeddingsnooverlap} when $\btau$ is any decorated tree and $\sigma=\small{\textcircled{\tiny{0}}}$. 
If $\phi_1$ is a graph embedding of $\tau$ in $T_n$, its image consists of $\lvert \tau \rvert$ vertices of $T_n$. 
Hence, the number of ways to embed $\sigma$ in $T_n$ without overlapping with $\phi_1(\tau)$ is $n-\lvert \tau \rvert$.
Thus 
\begin{align*}
    \bbE[\calF_{\btau,\small{\textcircled{\tiny{0}}}}(T_n)] 
    = \bbE[\calF_{\btau}(T_n)]  \cdot (n-\lvert \tau \rvert) 
    = \mathcal{O}\big(\bbE[\calF_{\btau}(T_n)] \cdot \bbE[\calF_{\small{\textcircled{\tiny{0}}}}(T_n)]\big),
\end{align*}
as required.

\paragraph*{Induction step}
Let $\btau,\bsigma$ be two decorated trees, both different from $\small{\textcircled{\tiny{0}}}$. 
Assume that~\eqref{asymptoticupperboundsecmomembeddingsnooverlap} holds for all pairs 
$(\boldsymbol{t},\boldsymbol{s})$ with either 
$\boldsymbol{t} \preceq \btau$ and $\boldsymbol{s} \prec \bsigma$ 
or $\boldsymbol{t} \prec \btau$ and $\boldsymbol{s} \preceq \bsigma$. 
In the following, we set:
\begin{align*}
	\omega_{n}^{(\btau,\bsigma)}:=\prod_{\ell=k}^{n} \Big(1+\frac{w(\btau)+w(\bsigma)}{(1+\alpha)\cdot \ell-2}\Big)^{-1}
\end{align*}
and
\begin{align*}%\label{sumforsigmaandtauinrecurrenceformulasecondmoment}
    \calS_{n}(\boldsymbol{t},\btau;\bsigma)
    := \sum_{\ell=k}^{n} \frac{1}{(1+\alpha)\ell-2} \cdot \omega_{\ell+1}^{(\btau,\bsigma)} \cdot \bbE[\calF_{\boldsymbol{t},\bsigma}(T_{\ell})]
\end{align*}
for any decorated tree $\boldsymbol{t}\prec\btau$. Iterating~\eqref{inequalityformularecpropsecmomembeddingsnooverlap}, we obtain:
\begin{align*}%\label{recurrenceformulagoodformsecondmoment}
	\bbE[\omega_{n+1}^{(\btau,\bsigma)} \cdot \calF_{\btau,\bsigma}(T_{n+1})] 
	&\leq \ \calF_{\btau,\bsigma}(S) 
	+ \sum_{\boldsymbol{t} \prec \btau} c^{*}(\boldsymbol{t},\btau) \cdot \calS_{n}(\boldsymbol{t},\btau;\bsigma) 
	+ \sum_{\boldsymbol{s} \prec \bsigma} c^{*}(\boldsymbol{s},\bsigma) \cdot \calS_{n}(\boldsymbol{s},\bsigma;\btau).
\end{align*}
Thus, to prove~\eqref{asymptoticupperboundsecmomembeddingsnooverlap}, it suffices to show that 
$\calS_{n}(\boldsymbol{t},\btau;\bsigma) = \mathcal{O}\big( \omega_{n}^{(\btau,\bsigma)} \cdot 
\bbE[\calF_{\btau}(T_n)] \cdot \bbE[\calF_{\bsigma}(T_n)] \big)$ for all $\boldsymbol{t} \prec \btau$ 
(by symmetry, the same will also hold for  $\calS_{n}(\boldsymbol{s},\bsigma;\btau)$ with $\boldsymbol{s} \prec \bsigma$).
Recall from Theorem~\ref{thm:first_moment} the asymptotic
\begin{align}\label{eq:omega_tau_sigma}
    \omega_{n}^{(\btau,\bsigma)} \cdot \bbE[\calF_{\btau}(T_n)] \cdot \bbE[\calF_{\bsigma}(T_n)] \big) 
    \approx  n^{\max\big\{0,1-\frac{w(\btau)}{1+\alpha}\big\}+\max\big\{0,1-\frac{w(\bsigma)}{1+\alpha}\big\}} 
    \cdot (\log n)^{\gamma(\btau)+\gamma(\bsigma)}.
\end{align}

Fix $\boldsymbol{t} \prec \btau$. According to the induction hypothesis and the above, 
the terms of $\calS_{n}(\boldsymbol{t},\btau;\bsigma)$ are bounded as 
\begin{align}\label{eq:S2}
	\frac{\omega_{n+1}^{(\btau,\bsigma)}}{(1+\alpha)n-2} \cdot \bbE[\calF_{\boldsymbol{t},\bsigma}(T_{n})] 
	= \mathcal{O}\Big( 
	n^{\max\big\{\frac{w(\boldsymbol{t})-w(\btau)}{1+\alpha},1-\frac{w(\btau)}{1+\alpha}\big\}+\max\big\{0,1-\frac{w(\bsigma)}{1+\alpha}\big\}-1} 
	 (\log n)^{\gamma(\boldsymbol{t})+\gamma(\bsigma)}\Big).
\end{align}
The sum of the above has different asymptotics depending on the value of the exponent of $n$:
\bigskip

\noindent\boxed{\textbf{1st case : $\max\big\{\tfrac{w(\boldsymbol{t})-w(\btau)}{1+\alpha},1-\tfrac{w(\btau)}{1+\alpha}\big\}+\max\big\{0,1-\tfrac{w(\bsigma)}{1+\alpha}\big\} < 0$}} \medskip \\
Then the sum of~\eqref{eq:S2} converges, thus 
$\calS_{n}(\boldsymbol{t},\btau;\bsigma) = \calO(1) = \mathcal{O}\big( \omega_{n}^{(\btau,\bsigma)} \cdot \bbE[\calF_{\btau}(T_n)] \cdot \bbE[\calF_{\bsigma}(T_n)] \big)$.
\bigskip 

\noindent\boxed{\textbf{2nd case : $\max\big\{\tfrac{w(\boldsymbol{t})-w(\btau)}{1+\alpha},1-\tfrac{w(\btau)}{1+\alpha}\big\}+\max\big\{0,1-\tfrac{w(\bsigma)}{1+\alpha}\big\} >0$}} \medskip \\
Then the sum of~\eqref{eq:S2} diverges, and a standard estimate provides the precise rate of growth:
\begin{align*}
\calS_{n}(\boldsymbol{t},\btau;\bsigma) 
= 
\mathcal{O}\Big( 
	n^{\max\big\{\frac{w(\boldsymbol{t})-w(\btau)}{1+\alpha},1-\frac{w(\btau)}{1+\alpha}\big\}+\max\big\{0,1-\frac{w(\bsigma)}{1+\alpha}\big\}} 
	 (\log n)^{\gamma(\boldsymbol{t})+\gamma(\bsigma)}\Big).
\end{align*}
Compare the above to~\eqref{eq:omega_tau_sigma} to find 
\begin{align*}
    \frac{\calS_{n}(\boldsymbol{t},\btau;\bsigma) }
    {\omega_{n}^{(\btau,\bsigma)}  \bbE[\calF_{\btau}(T_n)]  \bbE[\calF_{\bsigma}(T_n)]}
    =\mathcal{O}\Big(
     n^{\max\big\{\tfrac{w(\boldsymbol{t})-w(\btau)}{1+\alpha},1-\tfrac{w(\btau)}{1+\alpha}\big\}-\max\big\{0,1-\frac{w(\btau)}{1+\alpha}\big\}} 
    \cdot (\log n)^{\gamma(\boldsymbol{t}) - \gamma(\btau)}\Big).
\end{align*} 
Now, recall that 
\begin{tight_itemize}
\item $w(\boldsymbol{t}) \leq w(\btau)$ always;
\item if $w(\boldsymbol{t}) \leq  w(\btau)<1+\alpha$, then $\gamma(\boldsymbol{t}) = \gamma(\btau)= 0$; 
\item if $w(\boldsymbol{t}) =  w(\btau) \geq 1+\alpha$, then $\gamma(\boldsymbol{t})  < \gamma(\btau)$.
\end{tight_itemize} 
A separate analysis of the three different situations above shows that 
$\calS_{n}(\boldsymbol{t},\btau;\bsigma)= \mathcal{O}\big( \omega_{n}^{(\btau,\bsigma)} \cdot \bbE[\calF_{\btau}(T_n)] \cdot \bbE[\calF_{\bsigma}(T_n)] \big)$.
\bigskip 

\noindent\boxed{\textbf{3rd case : $\max\big\{\tfrac{w(\boldsymbol{t})-w(\btau)}{1+\alpha},1-\tfrac{w(\btau)}{1+\alpha}\big\}+\max\big\{0,1-\tfrac{w(\bsigma)}{1+\alpha}\big\} =0$}} \medskip \\
Then the sum of~\eqref{eq:S2} diverges, and a standard estimate provides the precise rate of growth:
\begin{align*}
    \calS_{n}(\boldsymbol{t},\btau;\bsigma) 
    = \mathcal{O}\Big( (\log n)^{\gamma(\boldsymbol{t})+\gamma(\bsigma)+1}\Big).
\end{align*}
By~\eqref{eq:omega_tau_sigma} and our assumption, we find 
\begin{align}\label{eq:S3}
    \frac{\calS_{n}(\boldsymbol{t},\btau;\bsigma) }
    {\omega_{n}^{(\btau,\bsigma)}  \bbE[\calF_{\btau}(T_n)]  \bbE[\calF_{\bsigma}(T_n)]}
%    &= \mathcal{O}\Big(
%     n^{-\max\big\{0,1-\frac{w(\btau)}{1+\alpha}\big\}-\max\big\{0,1-\frac{w(\bsigma)}{1+\alpha}\big\}} 
%    \cdot (\log n)^{\gamma(\boldsymbol{t}) - \gamma(\btau)+1}\Big) \\
    =\mathcal{O}\Big(
     n^{\max\big\{\tfrac{w(\boldsymbol{t})-w(\btau)}{1+\alpha},1-\tfrac{w(\btau)}{1+\alpha}\big\}-\max\big\{0,1-\frac{w(\btau)}{1+\alpha}\big\}} 
     (\log n)^{\gamma(\boldsymbol{t}) - \gamma(\btau)+1}\Big).
\end{align}
The power of $n$ in the right-hand side above is negative or null. 
When it is negative, equation~\eqref{eq:S3} is bounded, as required. 
It can only be $0$ in two cases: 
when $w(\btau) < 1+\alpha$ or when $w(\boldsymbol{t}) = w(\btau) \geq 1 + \alpha$. 
The former is incoherent with the assumption of this $3$rd case;
when the latter occurs, $\gamma(\boldsymbol{t}) \leq \gamma(\btau)-1$, hence~\eqref{eq:S3} is bounded. 
In conclusion,~\eqref{eq:S3} is always bounded, which is to say that 
$\calS_{n}(\boldsymbol{t},\btau;\bsigma)= \mathcal{O}\big( \omega_{n}^{(\btau,\bsigma)} \cdot \bbE[\calF_{\btau}(T_n)] \cdot \bbE[\calF_{\bsigma}(T_n)] \big)$.
\end{proof}

It remains to prove Proposition~\ref{propinequalityformularecpropsecmomembeddingsnooverlap}:

\begin{proof}[Proof of Proposition~\ref{propinequalityformularecpropsecmomembeddingsnooverlap}]
The strategy followed here is the same as for Proposition~\ref{prop:recurrence}. 
We start by fixing $\btau=(\tau,\ell_\tau)$ and $\bsigma=(\sigma,\ell_\sigma)$ two decorated trees, both different from $\small{\textcircled{\tiny{0}}}$. 
Recall our notation: in passing from $T_n$ to $T_{n+1}$ a new vertex denoted by $v_n$ is attached to a randomly chosen vertex $u_n$ of $T_n$. 
Our purpose is to compute the sum~\eqref{defobservablesecmomembeddingsnooverlap} over pairs of non-overlapping graph embeddings $\phi_1$, $\phi_2$ of $\tau$ and $\sigma$ into $T_{n+1}$.
We restrict the sum to embeddings with $\pi(\btau,\phi_1,T_{n+1}) > 0$ and $\pi(\bsigma,\phi_2,T_{n+1}) > 0$. 

As it has already been noted in the proof of Proposition~\ref{prop:recurrence}, such embeddings fall into three distinct categories: 
those which do not include $u_n$ nor $v_n$ in their image, 
those which include $u_n$ but not $v_n$ in their image 
and those which include both $u_n$ and $v_n$ in their image.
Embeddings that include $v_n$ but not $u_n$ in their image are not authorised 
as they only have positive weight when the decorated tree is $\small{\textcircled{\tiny{0}}}$.

A crucial remark is that two graph embeddings $\phi_1 : \tau \to T_{n+1}$, $\phi_2 : \sigma \to T_{n+1}$ with no overlap cannot be at once in one of the two last categories of embeddings. 
%Indeed, if $u_n \in \phi_1(\tau)$ - the same would symmetrically hold for $\sigma$ - then by the disjointness of $\phi_1(\tau)$ and $\phi_2(\sigma)$ and by the connectedness of $T_{n+1}$, $\tau$ and $\sigma$, either $\phi_2(\sigma)=\{v_n \}$ and $\sigma=\small{\textcircled{\tiny{0}}}$ or $\phi_2(\sigma) \subseteq T_{n+1}  \setminus \{u_n,v_n\}$. The first possiblity not being allowed by our hypothesis on $\bsigma$, $\phi_2$ falls necessarily in the first category of embeddings. 
We thus enumerate three kind of situations: 

\begin{enumerate}
\item the images of $\phi_1$ and $\phi_2$ are both in $T_{n+1}  \setminus \{u_n,v_n\}$;
\item the image of $\phi_1$ contains $u_n$ but not $v_n$, that of $\phi_2$ does not include any of them; or the same situation but reversing the roles of $\phi_1$ and $\phi_2$;
\item the image of $\phi_1$ contains $u_n$ and $v_n$, that of $\phi_2$ does not include any of them; or the same situation but reversing the roles of $\phi_1$ and $\phi_2$.
\end{enumerate}
Write $\calF^{(i)}_{\btau,\bsigma}(T_{n+1})$ with $i=1,2,3$ for the contribution to~\eqref{defobservablesecmomembeddingsnooverlap} of pairs of embeddings being in the corresponding situation above. 
As in the proof of Proposition~\ref{prop:recurrence}, the pairs of embeddings in the two first situations are in one-to-one correspondance with non-overlapping embeddings $\phi_1 : \tau \to T_n, \ \phi_2 : \sigma \to T_n$, 
although their weights $\pi(\btau,\phi_1,T_{n+1})$ and $\pi(\bsigma,\phi_2,T_{n+1})$ may differ from that in $T_n$. 
The same algebraic manipulations used for $\calF^{(1)}_{\btau}(T_{n+1})$ and $\calF^{(2)}_{\btau}(T_{n+1})$ in the proof of Proposition~\ref{prop:recurrence} may also be applied here to find
% $\calF^{(1)}_{\btau,\bsigma}(T_{n+1})$ and $\calF^{(2)}_{\btau,\bsigma}(T_{n+1})$ and lead us to find an analogue of the equation~\eqref{eq:F12}:
\begin{align}\label{ineqproofrecurrencesecmom}
	& \bbE \big[ \calF^{(1)}_{\btau,\bsigma}(T_{n+1}) +\calF^{(2)}_{\btau,\bsigma}(T_{n+1})\, \big|\, \calE_n\big]  \\
	&=	\Big[1 + \tfrac{|\ell_{\tau}|+|\ell_{\sigma}|}{(1+\alpha)n-2}\Big] \calF_{\btau,\bsigma}(T_{n}) +
		 \sum_{\tu \in \tau} \tfrac{\ell_{\tau}(\tu)(\ell_{\btau}(\tu) +\alpha -1)}{(1+\alpha)n-2}\cdot  \calF_{\btau^{\tu-},\bsigma}(T_n) +  \sum_{\tv \in \sigma} \tfrac{\ell_{\bsigma}(\tv)(\ell_{\bsigma}(\tv) +\alpha -1)}{(1+\alpha)n-2}\cdot  \calF_{\btau,\bsigma^{\tv-}}(T_n)\nonumber
\end{align}
Moreover,  the quantity $\calF^{(3)}_{\btau,\bsigma}(T_{n+1})$ may be treated as $\calF^{(3)}_{\btau}(T_{n+1})$ in~\eqref{eq:F3}, and we find
% In the same way, we deal with  $\calF^{(3)}_{\btau,\bsigma}(T_{n+1})$ as we did for $\calF^{(3)}_{\btau}(T_{n+1})$ minding if $u_n $ is in $\phi_1(\tau)$ or $\phi_2(\sigma)$, though. We show that:
\begin{align}\label{ineqproofrecurrencesecmom2}
	    &\bbE\Big[ \calF^{(3)}_{\btau,\bsigma}(T_{n+1}) \,\Big|\, \calE_n\Big]  \\
		&=
		\tfrac{1}{(1+\alpha)n-2} \cdot \bigg(
		\sum_{\tv \text{ l.l of} \ \btau} 
		\calF_{(\btau \setminus \tv)^+,\bsigma}(T_n) +
		(2\ell_{\tau}(\tu) +\alpha) \cdot \calF_{\btau \setminus \tv,\bsigma}(T_n) 
		 + \ell_{\tau}(\tu) (\ell_{\tau}(\tu) +\alpha -1) \cdot \calF_{(\btau \setminus \tv)^-,\bsigma}(T_n)  \nonumber \\ &+
		\sum_{\tv \text{ l.l of} \ \bsigma} 
		\calF_{\btau,(\bsigma \setminus \tv)^+}(T_n) +
		(2\ell_{\sigma}(\tu) +\alpha) \cdot \calF_{\btau,\bsigma \setminus \tv}(T_n) 
		 + \ell_{\sigma}(\tu) (\ell_{\sigma}(\tu) +\alpha -1) \cdot \calF_{\btau,(\bsigma \setminus \tv)^-}(T_n) \bigg)\nonumber
\end{align}
where the two sums are over the set of loose leaves of $\btau$ and $\bsigma$ respectively. The above equation is a derivative of~\eqref{eq:F3}. At the same stage, the last ingredient for the proof of Proposition~\ref{prop:recurrence} was a combinatorial lemma - Lemma~\ref{combinatoriallemma} - claiming that $\calF_{(\btau \setminus \tv)^+}(T_n)$ is actually a linear combination of $\calF_{\btau}(T_{n})$ and $\calF_{\btau \setminus \tv}(T_n)$ when $\tv$ is a loose leaf of $\btau$. Here that is almost what we do for $\calF_{(\btau \setminus \tv)^+,\bsigma}(T_n)$ and of course for $\calF_{\btau,(\bsigma \setminus \tv)^+}(T_n)$ by symmetry. The equality~\eqref{equalitycombinatoriallemma} is just turned into an inequality: 

\begin{lemma}\label{combinatoriallemmasecmom}
    Let $\btau=(\tau,\ell_{\tau})$ and  $\bsigma=(\tau,\ell_{\sigma})$ be two decorated trees, $\tv$ a \textit{loose leave} of $\btau$ and $\tu$ its only neighbour in $\tau$. Then, for any tree $T$:
    \begin{align*}%\label{equalitycombinatoriallemmasecmom}
    \calF_{\btau,\bsigma}(T) + C(\btau,w) \cdot \calF_{\btau \setminus \tv,\bsigma}(T) 
    \ \leq \ \calF_{(\btau \setminus \tv)^+,\bsigma}(T) 
    \ \leq \ \calF_{\btau,\bsigma}(T) + C^{*}(\btau,w) \cdot \calF_{\btau \setminus \tv,\bsigma}(T),
    \end{align*}
    with $C(\btau,w)$ being the constant defined in Lemma~\ref{combinatoriallemma} and $C^{*}(\btau,w):=C(\btau,w)+1$.
\end{lemma}

\begin{proof}
	As for Lemma~\ref{combinatoriallemma}, we proceed in two steps. 
	First, we express $\calF_{\btau,\bsigma}(T)$ in terms of pairs of non-overlapping embeddings of $\tau \setminus \tv$ and $\sigma$, respectively, in $T$.  
	
	To any non-overlapping pair of embeddings $\phi_1 : \tau \to T$, $\phi_2 : \sigma \to T$
	associate the restriction $\tilde\phi_1 : \btau \setminus\tv \to T$ of $\phi_1$ to the set of vertices of $\btau \setminus \tv$ 
	together with the same embedding $\phi_2$ of $\sigma$. 
	Conversely, if a pair of non-overlapping embeddings $\tilde\psi_1 : \tau \setminus\tv \to T$ and $\phi_2:\sigma\to T$ is given,  
	we may extend $\tilde \psi_1$ to some $\psi_1 : \tau \to T$ which preserves the non-overlapping property with $\phi_2$. 
	The number of such extensions depends on how many neighbours of $\tilde\psi_1(u)$ 
	are not contained in the images of $\tilde\psi_1$ and $\phi_2$. 
	The former occupies $\mathrm{deg}_{\tau \setminus\tv}(u)$ neighbours of $\tilde\psi_1(u)$; the latter may occupy $0$ or $1$ neighbour due to the non-overlapping requirement. 
%	The latter are always at most $\mathrm{deg}_{T}(\tilde\psi_1(u))-\mathrm{deg}_{\tau \setminus\tv}(u)$, but it can happen that some of those vertices belong to $\phi_2(\sigma)$. Actually, it turns out that only one vertex may find itself in such a position. Because if two of them were in it, then $\tilde\psi_1(u)$ would also be, since $\sigma$ and $T$ are trees (thus, acyclic connected graphs). What would contradict the hypothesis that $\tilde\psi_1(\tau \setminus \tv) \cap \phi_2(\sigma) = \emptyset$. 
	From this observation, and given that $\pi(\btau,\phi_1, T)=\pi(\btau\setminus \tv, \tilde\phi_1, T)$ ($\tv$ being a loose leaf), 
	we deduce that:
	\begin{align}\label{combinatorialequalityonelemmasecmom}
		\calF_{\btau,\bsigma}(T) 
		\leq \underset{\tilde\phi_1(\tau \setminus \tv)\cap\phi_2(\sigma)=\emptyset}{\sum_{\tilde\phi_1,\phi_2}} \,  	
		\Big( \mathrm{deg}_{T}(\tilde\phi_1(u))-\mathrm{deg}_{\tau \setminus \tv}(u) \Big) 
		\cdot \pi(\btau \setminus \tv,\tilde\phi_1, T) \cdot \pi(\bsigma,\phi_2, T)
\end{align}
	and 
	\begin{align}\label{combinatorialequalityonetwolemmasecmom}
		\calF_{\btau,\bsigma}(T) 
		\geq \underset{\tilde\phi_1(\tau \setminus \tv)\cap\phi_2(\sigma)=\emptyset}{\sum_{\tilde\phi_1,\phi_2}} \,  
		\Big( \mathrm{deg}_{T}(\tilde\phi_1(u))-\mathrm{deg}_{\tau \setminus \tv}(u) -1 \Big) 
		\cdot \pi(\btau \setminus \tv,\tilde\phi_1, T) \cdot \pi(\bsigma,\phi_2, T).
	\end{align}
The two above equations correspond to~\eqref{combinatorialequalityonelemma} in the proof of Lemma~\ref{combinatoriallemma}. Finally, similarly to~\eqref{combinatorialequalitytwolemma}, we find 
\begin{align}\label{combinatorialequalitytwolemmasecmom}
\calF_{(\btau \setminus \tv)^+,\bsigma}(T)  
	= \underset{\tilde\phi_1(\tau \setminus \tv)\cap\phi_2(\sigma)=\emptyset}{\sum_{\tilde\phi_1,\phi_2}} \, 
	\Big(  \mathrm{deg}_{T}(\tilde\phi_1(u))-1-\ell_{\tau}(u) \Big) \cdot \pi(\btau \setminus \tv,\tilde\phi_1, T) 
	\cdot \pi(\bsigma,\phi_2, T).
\end{align}
Subtracting~\eqref{combinatorialequalityonelemmasecmom} and~\eqref{combinatorialequalityonetwolemmasecmom} from~\eqref{combinatorialequalitytwolemmasecmom}, we obtain the desired lower and upper bounds, respectively.
%Then it is clear that the three last equations together implies~\eqref{equalitycombinatoriallemmasecmom} by using the same trick than at the very end of the proof of Lemma~\ref{combinatoriallemma}.
\end{proof}

To finish the proof of Proposition~\ref{propinequalityformularecpropsecmomembeddingsnooverlap}, insert
the upper bound of Lemma~\ref{combinatoriallemmasecmom} in~\eqref{ineqproofrecurrencesecmom2} to find 
\begin{align}\label{finalineqproofrecurrenceformulasecmom}
		&\bbE\Big[ \calF^{(3)}_{\btau,\bsigma}(T_{n+1}) \,\Big|\, \calE_n\Big]
		\leq
		\tfrac{\lvert \{\tv \ \text{loose leaf of} \ \btau\}\rvert + \lvert \{\tv \ \text{loose leaf of} \ \sigma\}\rvert}{(1+\alpha)n-2}
		\cdot
		\calF_{\btau,\bsigma}(T_n)  \\ 
		&+ \tfrac{1}{(1+\alpha)n-2} \sum_{\tv \text{ l.l. of } \btau} 
		(\mathrm{deg}_{\tau}(u)+\ell_{\tau}(u)+\alpha-1) \cdot \calF_{\btau \setminus \tv,\bsigma}(T_n) 
		 + \ell_{\tau}(\tu) (\ell_{\tau}(\tu) +\alpha -1) \cdot \calF_{(\btau \setminus \tv)^-,\bsigma}(T_n) \nonumber \\ 
		 &+\tfrac{1}{(1+\alpha)n-2} \sum_{\tv \text{ l.l. of } \bsigma} 
		(\mathrm{deg}_{\sigma}(u)+\ell_{\sigma}(u)+\alpha-1) \cdot \calF_{\btau, \bsigma \setminus \tv}(T_n) 
		 + \ell_{\sigma}(\tu) (\ell_{\sigma}(\tu) +\alpha -1) \cdot \calF_{\btau, (\bsigma \setminus \tv)^-}(T_n)\nonumber
	\end{align}
where the two sums are again over the set of loose leaves of $\btau$ and $\bsigma$ respectively. 
Equations~\eqref{ineqproofrecurrencesecmom} and~\eqref{finalineqproofrecurrenceformulasecmom} eventually lead to the expected inequality~\eqref{inequalityformularecpropsecmomembeddingsnooverlap}.
\end{proof}

\end{subsubsection}

\begin{subsubsection}{Proof of Theorem~\ref{thm:second_moment}}

Theorem~\ref{thm:second_moment} follows directly from the estimates of the two previous sections and from the decomposition 
\begin{align}\label{eq:F2}
	\calF_{\btau}(T_n)^2 = 
    \calF_{\btau,\btau}(T_n)
    +\calF_{\btau+\btau}(T_n).
\end{align}
 
\begin{proof}[Proof of Theorem~\ref{thm:second_moment}]
	Take the expectation of~\eqref{eq:F2} and insert the bounds~\eqref{eq:tau+tau} and~\eqref{eq:tautau}.
\end{proof}

\end{subsubsection}

\section{Observables and their difference around the seed tree}

\label{observablesandtheirdifferencearoundthegraphseed}

The goal of this section is to produce decorated trees $\btau$ that can distinguish between two different seeds $S$ and $S'$. 
Moreover, we wish that the asymptotic of $\calF_{\btau}(T_n^S)$ for such trees (both for the first and the square root of the second moment) 
be of the type $n^{\frac{w(\btau)}{1 + \alpha}}$, with a null logarithmic correction. 
The main result is the following. 

\begin{theorem}\label{thm:difference_from_seed}
	For any two seed trees $S \neq S'$ of size $k,k'\geq 3$, 
	there exists a decorated tree $\btau$ with $w(\btau) > 1 + \alpha$ and $\ell(u) \geq 2$ for all $u \in \tau$ such that 
	\begin{align*}%\label{estimatetheoremobservablesaroundtheseednonzeroasymptotically}
		\bbE[\mathcal{F}_{\btau}(T_{n}^{S})]-\bbE[\mathcal{F}_{\btau}(T_{n}^{S'})] 
		\approx n^{\frac{w(\btau)}{1+	\alpha}}.
	\end{align*}
\end{theorem}

The next four sections are concerned with the proof of Theorem \ref{thm:difference_from_seed}. 
To start, we will consider seeds $S$, $S'$ of same size. Finally, in Section \ref{sec:proof_main}, we prove our main result, Theorem \ref{thm:main}.

\subsection{Differences only appear around the seed}\label{afirstgeneralresultaroundtheseed}

We are there interested in the variation of the first moment of~\eqref{eq:observable}, 
depending on the value taken by the {seed tree} $S$. We state that :

\begin{proposition}\label{observablesaroundtheseedproposition1}
	For any two seed trees $S$, $S'$ of common size $k \geq 3$, 
	any decorated tree $\btau$ and any $n \geq k$ :
    \begin{align*}%\label{observablesaroundtheseedequality}
        \bbE[\mathcal{F}_{\btau}(T_{n}^{S})]-\bbE[\mathcal{F}_{\btau}(T_{n}^{S'})] 
        =  \bbE[\mathcal{F}_{\btau}(T_n^{S},\{\cdot\} \cap S \neq \emptyset)]-\bbE[\mathcal{F}_{\btau}(T_n^{S'},\{\cdot\} \cap S' \neq \emptyset)],
    \end{align*}
    where:
    \begin{align*}%\label{definitionobservablescapseed}
    \mathcal{F}_{\btau}(T_n^S,\{\cdot\} \cap S \neq \emptyset) 
    := \underset{\scriptstyle{\phi(\tau)\cap S \neq \emptyset }}{\underset{\phi: \tau \hookrightarrow T_{n}^{S}}{\sum}} 
    \prod_{u \in \tau} [\deg_{T_{n}^{S}}(\phi(u))-1]_{\ell(u)},
    \end{align*}
and the same for $S'$. In other words, $\mathcal{F}_{\btau}(T_n^S,\{\cdot\} \cap S \neq \emptyset)$ is the contribution to $\calF_{\btau}$ of embeddings that ``intersect the seed''.
\end{proposition}

The reader may be surprised that, in the statement above, we assume the location of the seed inside $T_n$ known, 
all while trying to prove that the seed may be determined.
It should be clear that, while Proposition~\ref{observablesaroundtheseedproposition1} and other steps of the proof of Theorem~\ref{thm:difference_from_seed} use the knowledge of the seed, 
their ultimate result (that is Theorem~\ref{thm:difference_from_seed}) does not. 

The proof is based on the coupling of Section~\ref{sec:coupling}, hence we use the planar $\alpha$-PA formalism. 

\begin{proof}
	Fix seed trees $S,S'$ of equal size $k \geq 3$. 
    We start by decomposing the sum over embeddings $\phi$ defining~\eqref{eq:observable} according to whether $\phi(\tau)\cap S$ is empty or not :
    \begin{align*}%\label{decomposeobservablesinmeanemptyornot}
    	\bbE[\mathcal{F}_{\btau}(T_n^S)] 
    	= \bbE[\mathcal{F}_{\btau}(T_n^S,\{\cdot\} \cap S \neq \emptyset)] 
    	+ \bbE[\mathcal{F}_{\btau}(T_n^S,\{\cdot\} \cap S = \emptyset)].
    \end{align*}
    If an embedding $\phi$ of $\tau$ in $T_{n}^{S}$ does not intersect $S$, 
    then $\phi(\tau)$ is necessarily a proper subset of one of the planted plane subtrees $T_{n}^{v,i}$, 
    and does not intersect the root of said tree. 
    Write $\calF_{\btau}(T_{n}^{v,i})$ for the contribution to $\calF_{\btau}(T_n^S)$ of all such embeddings. 
	Then
    \begin{align*}%\label{decomposeobservablesinmeanoutsidegraphseed}
    	\bbE[\mathcal{F}_{\btau}(T_n^S,\{\cdot\} \cap S = \emptyset)] 
	    =& \sum_{v \in V_S} \sum_{i=1}^{\deg_{S}(v)} \bbE[\calF_{\btau}(T_{n}^{v,i})].
    \end{align*}
    The right-hand side of the above does not actually depend on $S$, only on its size. 
    %Indeed, if $\lvert S \rvert = \rvert S' \lvert = k$, then $S$ and $S'$ have the same number of blue, resp. red corners, and we can use the coupling between $T_n^S$ and $T_{n}^{S',\alpha}$, as described at the end of section~\ref{sectiondecompositionintoplantedplanartreesandacoupling}. 
    Indeed, the coupling described at the end of Section~\ref{sec:coupling}
    indicates that the same planted plane subtrees may be used to construct $T_n^{S}$ and $T_n^{S'}$. 
    While it is not necessary for this argument, one may notice that $T_n^{S}$ and $T_n^{S'}$ may even be coupled so that 
    $$ \mathcal{F}_{\btau}(T_n^S,\{\cdot\} \cap S = \emptyset)
    = \mathcal{F}_{\btau}(T_n^{S'},\{\cdot\} \cap S' = \emptyset) \qquad \text{a.s.}$$
    The result follows readily. 
\end{proof}

\subsection{Blind trees: differences appear only in the seed}\label{sec:blindtreesandarefinement}

Next we aim to improve Proposition~\ref{observablesaroundtheseedproposition1}
by showing that $\bbE[\mathcal{F}_{\btau}(T_{n}^{S})]-\bbE[\mathcal{F}_{\btau}(T_{n}^{S'})]$ only depends on embeddings totally
contained in the seed, not just intersecting it. 
This will not be true for all trees $\btau$, only for special ones. 
Some definitions are required. 

\begin{definition}[Perfect embedding]\label{decoratedtree}
    Let $\btau = (\tau,d)$ be a decorated tree and $T$ a (larger) tree. 
   An embedding $\phi: \tau \hookrightarrow T$ is called {\bf perfect} if
    \begin{align*}%\label{degreeequalityfulldecoratedembedding}
    	d(u) = \deg_{T}(\phi(u)) \qquad \forall u \in \tau.
    \end{align*}
\end{definition}

Write $D_{\tau,d}(T)$ for the number of perfect embeddings from $(\tau,d)$ to $T$. 
Note that any embedding of $\tau$ in $T$ is perfect for exactly one decoration, hence contributes exactly to one $D_{\tau,d}(T)$.

\begin{definition}[Blind tree]\label{blindtreedefinition}
    Let $\tau$, $T_{1}$ and $T_{2}$ be finite trees with $|T_1| = |T_2|$.
    We say that $\tau$ is a $\mathbf{(T_{1},T_{2})}$\textbf{-blind tree} if for any decoration $d : \tau \to \mathbb{N}^*$ :
    \begin{align*}%\label{equalityblindtree}
    	D_{\tau,d}(T_1) = D_{\tau,d}(T_2).
    \end{align*}
\end{definition}

Intuitively, $\tau$ is $(T_{1},T_{2})$-blind if it can not distinguish between $T_1$ and $T_2$ using observables involving the number of embeddings and the environment around these embeddings (as are our observables $\calF_{\btau}$). 
For instance, the tree formed of a single vertex is $(T_{1},T_{2})$-blind if and only if $T_1$ and $T_2$ have the same degree sequence. 

Let $(\tau,d)$ be a decorated tree,  $\sigma \subseteq \tau$ a non empty subtree of $\tau$, and $S$ be a seed tree. 
Write $D_{\sigma,\tau,d}(T_{n}^{S})$ for the number of perfect embeddings $\Phi$ of $(\tau,d)$ in $T_{n}^{S}$
with $\Phi(\tau) \cap S = \Phi(\sigma)$. 
We obviously have :
\begin{align}\label{eq:F_intersect_S}
	\mathcal{F}_{\btau}(T_n^{S},\{\cdot\} \cap S \neq \emptyset) 
	=  \sum_{\sigma \subseteq \tau} \sum_{d :  \tau \to \mathbb{N}^*} 
	\ D_{\sigma,\tau,d}(T_{n}^{S}) \cdot \prod_{u \in \tau} \ [d(u)-1]_{\ell(u)}.
\end{align}

The next proposition constitutes the essential step for the upgrade of Proposition~\ref{observablesaroundtheseedproposition1}.

\begin{proposition}\label{propositionblindtreePAmodel}
	Let $S$ and $S'$ be two seed trees of common size $k \geq 3$. For trees $\sigma \subseteq \tau$ with $\sigma$ which is $(S,S')$-blind,
	\begin{align*}%\label{equalitypropositionblindtreePAmodel}
		\forall \ d : \tau \to \mathbb{N}^*\text{ and } n\geq k, \qquad 
		\bbE[D_{\sigma,\tau,d}(T_{n}^{S})] = \bbE[D_{\sigma,\tau,d}(T_{n}^{S'})].
	\end{align*}
\end{proposition}

\begin{proof}
    The idea of the proof is to decompose the embeddings contributing to $D_{\tau,\sigma,d}(T_{n}^{S})$ 
    according to degrees of the vertices belonging to its image {\em in the seed tree $S$}. 
    For illustration, we will start with the simpler case when $\sigma = \tau$, then move on to the general case. \\
    \\
    \noindent\boxed{\textbf{Particular case $\sigma=\tau$}} \\
    \\
        For any seed $S$ and $n \geq |S|$,
        \begin{align*}
        	D_{\tau,\tau,d}(T_{n}^{S})
        	= \sum_{\phi: \tau \hookrightarrow S}\,\, \prod_{\tu \in \tau}\mathbf{1}_{\deg_{T_{n}^{S}}(\phi(\tu))=d(\tu)}.
        \end{align*}
        Due to the coupling explained in Section~\ref{sec:coupling}, 
        the probability of the event $\bigcap_{\tu \in \tau}\{\deg_{T_{n}^{S}}(\phi(\tu))=d(\tu)\}$
        only depends on the degrees of the vertices $(\phi(\tu))_{\tu \in \tau}$ in $S$. 
        Write $f_{d}[\deg_S\phi(\tu):\, \tu \in \tau]$ for this probability. 
        Then 
        \begin{align}\label{eq:embeddings_in_seed}
        	\bbE[D_{\tau,\tau,d}(T_{n}^{S})]
        	= \sum_{\phi: \tau \hookrightarrow S} f_d[\deg_S\phi(\tu):\, \tu \in \tau]
        	= \sum_{\ell:\tau \to \bbN^*} D_{\tau,\ell}(S) \cdot f_d[\ell(\tu):\, \tu \in \tau].
        \end{align}
        Since $\tau$ is assumed $(S,S')$-blind, the quantity above remains unaltered when $S$ is replaced by $S'$. \\
    \\
    \noindent\boxed{\textbf{General case $\sigma\subset\tau$}} \\
    \\
        For a seed $S$ and $v \in S$, write $(T_n^{S}\setminus S)_v$ for the subtree of $T_n^S$ 
        formed of all vertices that may be connected to $v$ without using any edge of $S$. 
        In terms of vertices $\{(T_n^{S}\setminus S)_v: v\in S\}$ is a partition of the vertices of $T_n^S$; 
        in terms of edges, it is a partition of the edges of $T_n^S$ not contained in  $S$. 
        In the language of Section~\ref{sec:coupling}, $(T_n^{S}\setminus S)_v$ 
        is the tree obtained by gluing $T_n^{v,1},\dots,T_n^{v,\deg_S(v)}$ together at the root. 
        Use the same notation for $\tau$ and $\sigma$:
        for $\tv \in \sigma$ write $(\tau \setminus \sigma)_\tv$ for the subtree of $\tau$ 
        formed of all vertices that may be connected to $\tv$ in $\tau$ without using any edge of $\sigma$.  
        
        For a decoration $d$ of $\tau$, and vertices $\tv \in \sigma$ and $v \in S$, write 
        $\calD(d,\tv, v, T_n^S)$ for the number of perfect embeddings $\phi$ 
        of $((\tau \setminus \sigma)_\tv, d)$  in $(T_n^{S}\setminus S)_v$
        with $\phi(\tv) = v$: 
        \begin{align*}
        	\calD(d,\tv, v, T_n^S) =
        	\sum_{\phi :(\tau \setminus \sigma)_\tv \to (T_n^{S}\setminus S)_v} 
        	\ind_{\phi(\tv) = v} \prod_{\tu \in (\tau \setminus \sigma)_\tv } \ind_{\{\deg_{T_n^S}(\phi(\tu)) = d(\tu)\}}.
        \end{align*}
        It may not be explicit in the above, but the last product does only depend on $(T_n^{S}\setminus S)_v$ for $\tu \neq \tv$, since $\deg_{T_n^S}(\phi(\tu))=\deg_{(T_n^{S}\setminus S)_v}(\phi(\tu))$. For $v \in S$, it should be noted that~$\deg_{T_n^S}(v)=\deg_{(T_n^{S}\setminus S)_v}(v)+\deg_{S}(v)$.
        
        Now express $D_{\sigma,\tau,d}(T_{n}^{S})$ as follows: 
        \begin{align}\label{eq:calD}
        	D_{\sigma,\tau,d}(T_{n}^{S})
        	= \sum_{\phi:\sigma \hookrightarrow S} \,\,\prod_{\tv \in \sigma} \calD(d,\tv, \phi(\tv), T_n^S).
        \end{align}
        Indeed, any embedding contributing $D_{\tau,\sigma,d}(T_{n}^{S})$ comes from an embedding of $\sigma$ in $S$ and 
        embeddings of the subtrees $(\tau \setminus \sigma)_\tv$ in the corresponding subtrees of $T_n^{S}\setminus S$.
        The conditions that the embedding is perfect may be verified separately in each subtree. 
        
        Now, due to the coupling of Section~\ref{sec:coupling} and 
        to the expression of $(T_n^{S}\setminus S)_v$ in terms of $(T_n^{v,i})_{i\geq1}$, 
        the expectation of the product above only depends on the degrees of the vertices $\phi(\tv)$ in $S$:
        \begin{align*}
        	\bbE\Big[\prod_{\tv \in \sigma} \calD(d,\tv, \phi(\tv), T_n^S)\Big] = f_{\sigma,\tau,d}[\deg_S(\phi(\tv)) : \tv \in \sigma], 
        \end{align*}
        for some explicit function $f_{\sigma,\tau,d}$. Injecting this in~\eqref{eq:calD}, we find
        \begin{align*}
        	\bbE\big[D_{\sigma,\tau,d}(T_{n}^{S})\big]
        	= \sum_{\phi:\sigma \hookrightarrow S} \,f_{\sigma,\tau,d}[\deg_S(\phi(\tv)) : \tv \in \sigma]
        	= \sum_{\ell:\sigma \to \bbN^*} D_{\sigma,\ell}(S) \cdot f_{\sigma,\tau,d}[\ell(\tv) : \tv \in \sigma].
        \end{align*}
		Since $\sigma$ is assumed $(S,S')$-blind, the quantity above is equal when $S$ is replaced by $S'$. 
\end{proof}

We are now ready to give a finer version of Proposition~\ref{observablesaroundtheseedproposition1}. 
Consider two distinct seeds $S,S'$ of equal size $k \geq 4$ (non two distinct seeds of smaller size exist). 
Then there exists at least one tree which is not $(S,S')$-blind, for instance $S$ or $S'$ have this property. 
It follows that there exists at least one minimal tree which is not $(S,S')$-blind, 
that is a tree $\tau$ which is not $(S,S')$-blind but for which any proper subtree $\sigma \subsetneq \tau$ is $(S,S')$-blind. 
%Within this framework, we obtain a significant improvement for the equation~\eqref{observablesaroundtheseedequality} :

\begin{corollary}\label{cor:difference_in_seed}
    Let $\tau$ be a minimal tree that is not ${(S,S')}$-blind. 
    Then for any decoration $\ell$ of~$\tau$    
    \begin{align*}%\label{observablesaroundtheseedequalityminimalnonblindtree}
	    \bbE[\mathcal{F}_{\btau}(T_n^S)]-\bbE[\mathcal{F}_{\btau}(T_n^{S'})] 
	    =  \bbE[\mathcal{F}_{\btau}(T_n^S,\{\cdot\} \subseteq S)]-\bbE[\mathcal{F}_{\btau}(T_n^{S'},\{\cdot\} \subseteq S')],
    \end{align*}
    where
    \begin{align*}%\label{definitionobservablesembeddedinseed}
        \mathcal{F}_{\btau}(T_n^{S},\{\cdot\} \subseteq S) 
        := \sum_{\phi: \tau \hookrightarrow S} \,\, \prod_{u \in \tau} [\deg_{T_n^S}(\phi(u))-1]_{\ell(u)},
    \end{align*}
and the same for $S'$.
\end{corollary}

\begin{proof}
    We apply~\eqref{eq:F_intersect_S} together with Proposition~\ref{propositionblindtreePAmodel}
    and use the minimality of $\tau$ to obtain
    \begin{align*}
    	& \bbE[\mathcal{F}_{\btau}(T_n^S,\{\cdot\} \cap S \neq \emptyset)]-\bbE[\mathcal{F}_{\btau}(T_n^{S'},\{\cdot\} \cap S' \neq \emptyset)]  \\ 
    	& \qquad =  \sum_{d :  \tau \to \mathbb{N}^{*}} \ 
    	\Big( \prod_{u \in \tau} [d(u)-1]_{\ell(u)} \Big) \cdot 
    	\big(\bbE[D_{\tau,\tau,d}(T_n^S)]-\bbE[D_{\tau,\tau,d}(T_{n}^{S'})]\big).
    \end{align*}
    Furthermore, it is clear that:
    \begin{align*}%\label{eq:F_in_S}
        \bbE[\mathcal{F}_{\btau}(T_n^{S},\{\cdot\} \subseteq S)] 
        = \sum_{d :  \tau \to \mathbb{N}^{*}} \  \Big( \prod_{u \in \tau} [d(u)-1]_{\ell(u)} \Big) \cdot \bbE[D_{\tau,\tau,d}(T_n^S)].
    \end{align*}
    Of course, the above equality also holds when $S$ is replaced by $S'$. Hence the result.
\end{proof}

\subsection{Evaluating the difference in the seed}

In light of the above, the quantities of interest for the proof of Theorem~\ref{thm:difference_from_seed}
are of the type $\bbE[\mathcal{F}_{\btau}(T_n^S,\{\cdot\} \subseteq S)]$.
We give below a more convenient expression for them based on~\eqref{eq:embeddings_in_seed}.

\newcommand{\other}{{\rm other}}
\begin{lemma}\label{lemmaobservablesinseed}
	Let $S$ be a seed tree of size $k$ and $\btau = (\tau,\ell)$ a decorated tree. Then, for any $n \geq k$
	\begin{align}\label{lemmaobservablesinseedequality}
		\bbE[\mathcal{F}_{\btau}(T_n^S,\{\cdot\} \subseteq S)] 
		= \sum_{d : \tau \to \mathbb{N}^*} \mathbf{f}(k,n;\, d,\ell) \cdot D_{\tau,d}(S),
	\end{align}
	where $\mathbf{f}(k,n;\, d,\ell)$ are functions defined as follows (for decorations $d$ such that $|d| < 2k-2$).
	
	%\begin{tight_itemize}
	%\item 
	Let $(y_n(u) : u \in \btau \cup \{\other\}, \, n \geq k)$ be a P\'olya urn with $|\tau| +1$ colours, 
	replacement matrix $M=(m_{uv})_{u,v \in  \btau \cup \{\other\}}$ with 
	\begin{align*}
    	m_{uv} 
    	= \begin{cases}
    	1 \qquad &\text{ if $u = v \neq \other$},\\
    	1 + \alpha &\text{ if $u=v = \other$}, \\
    	\alpha &\text{ if $u\neq v = \other$}, \\	
    	0 &\text{otherwise},	
    	\end{cases}
    \end{align*}
	and initial states $y_k(u) = \alpha + d(u)-1$ for $u \in \tau$ and $y_k(\other) = (1+\alpha) k - 2 - \sum_{u \in\tau} y_k(u)$.
	%\item 
	Then 
	$$\mathbf{f}(k,n;\, d,\ell) =\bbE\Big[ \prod_{u\in \tau}\, [y_{n}(u) - \alpha ]_{\ell(u)}\Big].$$
%\end{tight_itemize}
\end{lemma}

\begin{remark}
	The actual definition of the variables $(y_n(u) : u \in \btau \cup \{\other\}, \, n \geq k)$ is not very important; 
	a more intuitive expression will be used (see~\eqref{eq:f_deg}).
	The important aspect of~\eqref{lemmaobservablesinseedequality} is that $\bbE[\mathcal{F}_{\btau}(T_n^S,\{\cdot\} \subseteq S)]$ 
	is factorised between a part that depends on $n$ but not the seed structure (namely $\mathbf{f}(k,n;\, d,\ell)$) 
	and one that depends on the seed structure but not on $n$ (namely $D_{\tau,d}(S)$). 
\end{remark}

\begin{proof}
	Fix $S$ and $\btau = (\tau,\ell)$.
	% and $d : \tau \to \bbN^*$ as in the proposition. 
	Recall that 
	\begin{align}\label{eq:in_seed_1}
		\bbE[\mathcal{F}_{\btau}(T_n,\{\cdot\} \subseteq S)] 
		= \sum_{\phi: \tau \hookrightarrow S} \bbE\Big[ \prod_{u \in \tau} \big[\deg_{T_n^S}(\phi(u)) -1 \big]_{\ell(u)}\Big],
	\end{align}
	where the sum is over all embeddings of $\tau$ in $S$.
	
	Let $\phi$ be an embedding of $\tau$ in $S$. 
	Then the family $\{\deg_{T_n^S}(\phi(u)): \, u \in \tau\}$ has a markovian dynamics as $n$ increases described as follows:
	\begin{tight_itemize}
		\item for each $u \in \tau$, with probability $\frac{\deg_{T_n^S}(\phi(u)) -1 + \alpha}{(1+\alpha)n -2}$, we have
		$\deg_{T_{n+1}^S}(\phi(u)) = \deg_{T_n^S}(\phi(u)) +1$ and all other entries remain the same; 
		\item otherwise $\deg_{T_{n+1}^S}(\phi(u)) = \deg_{T_n^S}(\phi(u))$ for all $u \in \tau$. 
	\end{tight_itemize}
 	Then, if we set 
	\begin{align*}
		&y_{n}(u) = \deg_{T_{n}^S}(\phi(u)) +\alpha - 1, \qquad\qquad \text{ for $u \in \tau$ and }\\
		&y_{n}(\other) =  (1+\alpha) n - 2 - \sum_{u \in\tau} y_n(u) =  \sum_{v \in T_n^S \setminus \phi(\tau)} \deg_{T_{n}^S}(v) +\alpha - 1,
	\end{align*}
	we deduce readily that the family $(y_n(u) : u \in \btau \cup \{\other\}, \, n \geq k)$ has the dynamics
	of a P\'olya urn with the replacement matrix $M$ and initial conditions given in the statement. 
%	Moreover, its initial condition is $y_k(u) = \alpha + d(u)-1$ for $u \in \tau$ 
%	and $y_k(\other) = (1+\alpha) k - 2 - \sum_{u \in\tau} y_k(u)$.
	Thus 
	\begin{align}\label{eq:f_deg}
		\bbE\Big[ \prod_{u \in \tau} \big[\deg_{T_n^S}(\phi(u)) -1 \big]_{\ell(u)}\Big] 
		=\bbE\Big[ \prod_{u\in \tau}\, [y_{n}(u) - \alpha ]_{\ell(u)}\Big]
		=\mathbf{f}(k,n;\, d,\ell).
	\end{align}
	
	By inserting the above in~\eqref{eq:in_seed_1} and grouping the terms of the sum by the degrees in $S$ of the embedding, 
	we obtain~\eqref{lemmaobservablesinseedequality}.
\end{proof}
%
%\begin{remark}\label{remarkdefquantitiesf}
%Equation~\eqref{defquantitiesf} could also be written with any set of $\lvert \tau \rvert$ vertices belonging to $S$ since the adjacency relations between them play no role in the law of their degrees. Indeed, according to the transitions probabilities~\eqref{probabilitylawvertexaffinePAmodel}, the degrees of the $\lvert \tau \rvert$ chosen vertices within the seed tree together behave as the proportions of any $\lvert \tau \rvert$ colors of balls selected among the $k$ colors of a standard Polya urn with affine reinforcement and whose initial proportions were $\deg_{S}(u)$ for $u \in S$.
%\end{remark}

Corollary~\ref{cor:difference_in_seed} and Lemma~\ref{lemmaobservablesinseed} state that, for seeds $S \neq S'$ of same size and $\tau$ 
which is a minimal non-$(S,S')$-blind tree, the difference of the observables $\calF_{\btau}$ for $T_n^S$ and $T_n^{S'}$ 
is a linear combination of functions $\mathbf{f}(k,n;\, d,\ell)$:
\begin{align}\label{corollaryemmaobservablesinseedequality}
	\bbE\big[\mathcal{F}_{\btau}(T_n^S)\big]-\bbE\big[\mathcal{F}_{\btau}(T_n^{S'})\big]  
	=& \sum_{d: \tau \to \bbN^*} \mathbf{f}(k,n;\, d,\ell) \cdot [D_{\tau,d}(S)-D_{\tau,d}(S')]
\end{align}
for any decoration $\ell$ of $\tau$. 
Thus, for our proof of Theorem~\ref{thm:difference_from_seed}, 
it will be of great interest to study the asymptotics of the functions $ \mathbf{f}(k,n;\, d,\ell)$ as $n \to \infty$. 
The relevant result is the following.

\begin{proposition}\label{asymptoticsonquantitiesf}
    Let $S$ be a seed tree of size $k \geq 2$, $\btau = (\tau,\ell)$ be a decorated tree 
    and $d : \tau \to \bbN^*$ be some decoration of $\tau$. 
    Then
    \begin{align}\label{convergencerescaledquantitiesf}
    	n^{-\frac{|\ell|}{1+\alpha}} \cdot \mathbf{f}(k,n;\, d,\ell)
    	\xrightarrow[n \to \infty]{} \ C(k,|\ell|) \cdot \prod_{u \in \tau}\ [d(u) +\ell(u)+\alpha-2]_{\ell(u)},
    \end{align}
    where $C(k,|\ell|) > 0$ is some constant depending only on $k$ and $|\ell|$. 
\end{proposition}

The above is a technical result based on the study of P\`olya urns;
the rest of the section is dedicated to proving it. 
It is possible to prove~\eqref{convergencerescaledquantitiesf} using only the abstract quantities $(y_n(u) : u \in \btau \cup \{\other\}, \, n \geq k)$. We prefer however to use the more visual interpretation of these quantities in terms of  the degrees in $T_n^S$ of some embedding of $\tau$ in $S$.

%We defer its proof to the next section and show now how it implies Theorem~\ref{thm:main}. 

% Rather than working with the abstract definition of $\mathbf{f}(k,n;\, d,\ell)$, we will work with its interpretation. 

Fix a seed $S$ of size $k$ and $R$ a subtree of $S$. 
Also fix some decoration $\ell$ of $R$.
Set 
\begin{align*}
	M_n^{\ell} 
	= \prod_{u \in R}\, [\deg_{T_n^S} (u) +\alpha + \ell(u) - 2]_{\ell(u)}
	% = \prod_{u \in R}\, [y_n(u) + \ell(u) -1]_{\ell(u)}
\end{align*}
and
\begin{align*}
    	\calW_n^{\ell}
    	= \begin{cases}
    	1 \qquad &\text{ if $n=k$},\\
    	 \displaystyle{\prod_{t=k}^{n-1} \Big( 1 + \frac{|\ell|}{(1+\alpha)t-2} \Big)^{-1}} &\text{ if $n \geq k+1$}.	
    	\end{cases}
    \end{align*}
When the dependance on $\ell$ is not important, we will drop it from the notation.

\begin{lemma}
	The sequence $(M_n \cdot \calW_n )_{n\geq k}$ is a martingale and is bounded in $L^2$. 
\end{lemma}

\begin{proof}
	We start by proving that $(M_n  \cdot \calW_n )_{n\geq k}$ is a martingale. We introduce $\calE_n$ the $\sigma$-algebra generated by $M_{k},\dotsc,M_n$. 
	When going from $n$ to $n+1$ one of two things can happen: 
	either the new vertex of $T_{n+1}^{S}$ is attached to a vertex $u \in  R$ or it is attached to some other vertex of $T_n^S$. 
	In the first case $M_n$ is multiplied by $\frac{\deg_{T_n^S} (u) +\alpha + \ell(u) -1 }{\deg_{T_n^S} (u) +\alpha -1}$;
	this occurs with probability $\frac{1}{(1+\alpha)n - 2} (\deg_{T_n^S} (u) +\alpha -1)$. 
	In the latter case the value of $M_n$ remains unchanged.  
	Thus
	\begin{align*}
		\frac{\bbE(M_{n+1}| \calE_n) - M_n}{M_n}
		= \sum_{u \in R}  
		\frac{\deg_{T_n^S} (u) +\alpha -1}{(1+\alpha)n - 2} 
		\cdot \Big(\frac{\deg_{T_n^S} (u) +\alpha -1 + \ell(u)}{\deg_{T_n^S} (u) +\alpha - 1} - 1\Big)
%		&= \frac{M_n}{(1+\alpha)n - 2} \sum_{u \in R}
%		\ell(u)
		= \frac{|\ell|}{(1+\alpha)n - 2}.
	\end{align*}
	It follows that $\bbE(M_{n+1}| \calE_n) = \big(1+ \frac{|\ell|}{(1+\alpha)n - 2}\big)\cdot M_n$, 
	which is to say that $M_n \cdot \calW_n$ is a martingale. \\

	That the martingale $M_n  \cdot \calW_n$ is bounded in $L^2$ follows from:
\begin{align*}
\forall u \in R, \ [d+\alpha+\ell(u)-2]_{\ell(u)}^{2} \leq [d+\alpha+2\ell(u)-2]_{2\ell(u)}
\end{align*}
and also:
\begin{align*}
(\calW_n^{\ell})^{2} \leq \calW_n^{2 \ell},
\end{align*}
where $2 \ell$ is the decoration such that for every $u \in R$, $(2 \ell) (u) = 2 \cdot \ell(u)$.
Indeed, the two above inequalities together imply that $(M_n^{\ell}  \cdot \calW_n^{\ell})^{2} \leq M_n^{2\ell}  \cdot \calW_n^{2\ell}$ for every $n \geq k$. But since $M_n^{2\ell}  \cdot \calW_n^{2\ell}$ is a martingale according to what we have just proved, it is bounded in $L^1$, thus $M_n  \cdot \calW_n$ is for its part bounded in $L^2$.
\end{proof}

\begin{corollary}\label{cor:martingale_cv}
	The following convergence holds almost surely and in $L^1$:
	\begin{align}\label{eq:as_cv_tau}
		 n^{-\frac{|\ell|}{1 + \alpha}}  \cdot M_n\xrightarrow[n \to \infty]{}C(k,|\ell|) \cdot \xi(R),
	\end{align}
	where $\xi(R)$ is a random variable of expectation $\displaystyle\bbE[\xi(R)] = \prod_{u \in R} [\deg_S(u) + \ell(u) +\alpha -2]_{\ell(u)}$
	and $C(k,|\ell|)  > 0$ is a universal constant depending only on $k$ and $|\ell|$.
\end{corollary}

\begin{proof}
	As a martingale that is bounded in $L^2$, $(M_n \cdot \calW_n)$ 
	converges a.s. and in $L^1$ when $n\to \infty$ to some random variable $\xi(R)$.
	By the $L^1$ convergence, 
	\begin{align*}
		\bbE[\xi(R)] 
		= \bbE[M_k \cdot \calW_k] 
		= \prod_{u \in R} [\deg_S(u) + \ell(u) +\alpha -2]_{\ell(u)}.
	\end{align*}
	Finally, as in~\eqref{polynomialgrowthfactor}, a straightforward computation proves that  
	$\calW_n \cdot n^{\frac{|\ell|}{1 + \alpha}}$ converges as $n\to\infty$ to some constant depending only on $|\ell|$ and $k$. 
	This implies~\eqref{eq:as_cv_tau}.  
\end{proof}

\begin{remark}\label{rem:xi(u)}
	If we apply Corollary~\ref{cor:martingale_cv} to $R$ being formed of a single vertex $u$ with $\ell(u) = 1$, we obtain 
	\begin{align*}
		n^{-\frac{1}{1 + \alpha}}  \deg_{T_n^S}(u) \xrightarrow[n \to \infty]{} C(k,1) \cdot \xi(u), 
	\end{align*}
	where $\bbE[\xi(u)] = \deg_S(u) + \alpha -1$ and $C(k,1) > 0$ is defined as above. 
%	This is the asymptotic announced in \eqref{convergencedegreesinaffinePAtrees}.
	Moreover, it may be shown that $\xi(u) >0$ a.s. (see \cite[Lemma 8.17]{hofstad_2016}), which is to say that $\deg_{T_n^S}(u)$ grows as $n^{\frac{1}{1 + \alpha}}$.
\end{remark}

We are now ready to prove Proposition~\ref{asymptoticsonquantitiesf}.

\begin{proof}[Proof of Proposition~\ref{asymptoticsonquantitiesf}]
	Fix some $S$, $\tau$, $\ell$ and $d$ as in the proposition. 
	Let $\phi$ be some perfect embedding of $(\tau,d)$ in $S$. 
	Write $R = \phi(\tau)$ and keep in mind that $\deg_S(\phi(u)) = d(u)$ for all $u \in \tau$.
	We will use the notation of Corollary~\ref{cor:martingale_cv} for $R$.  
	
	Using~\eqref{eq:f_deg} and the fact that $\deg_{T_n^S} (u) \to \infty$ for all $u \in \tau$, we have
	\begin{align*}
		\frac{ \prod_{u\in \tau}\, [y_{n}(u) - \alpha ]_{\ell(u)}}{M_n}
		= \prod_{u \in \tau} \frac{[\deg_{T_n^S}(\phi(u))-1]_{\ell(u)}}{[\deg_{T_n^S}(\phi(u))+ \ell(u) +\alpha -2]_{\ell(u)}}
		\xrightarrow[n\to\infty]{} 1. 
	\end{align*}
	Moreover, the ratio is always bounded from above by $1$%
	\footnote{This is proved individually for each term in the product: 
	it is clear when $\ell(u) \geq 1$; when $\ell(u) = 0$, the ratio is $1$}.
	Thus, by the dominated convergence theorem, 
	\begin{align*}
		\lim_{n \to \infty}n^{-\frac{|\ell|}{1 + \alpha}}\cdot  \mathbf{f}(k,n;\, d,\ell)
		= \lim_{n \to \infty}n^{-\frac{|\ell|}{1 + \alpha}}  \cdot \bbE [M_n]		
		= C(k,|\ell|) \cdot \prod_{u \in \tau} [d(u) + \ell(u) +\alpha -2]_{\ell(u)},
	\end{align*}
	where the last equality is due to Corollary~\ref{cor:martingale_cv}.
\end{proof}

\subsection{Proof of Theorem~\ref{thm:difference_from_seed}}

In proving Theorem~\ref{thm:difference_from_seed} we analyse differently the case where the two seeds $S$, $S'$ have same size and that where they have distinct sizes. We start with the former.

\begin{proof}[Proof of Theorem~\ref{thm:difference_from_seed} for seeds of common size]
    Fix $S,S'$ two distinct seeds of same size $k \geq 4$. 
    Let $\tau$ be a minimal tree which is not $(S,S')$-blind.
    Then, for any decoration $\ell$ of $\tau$, 
    equation~\eqref{corollaryemmaobservablesinseedequality} and Proposition~\ref{asymptoticsonquantitiesf} imply that
    \begin{align}\label{eq:f_D}
     	\lim_{n \to \infty} n^{-\frac{|\ell|}{1+\alpha}} \cdot  
    	\big( \bbE[\mathcal{F}_{\btau}(T_n^{S})]-\bbE[\mathcal{F}_{\btau}(T_n^{S'})] \big) 
    	= \sum_{d: \tau \to \bbN^*} \mathbf{f}(\infty;\, d,\ell) \cdot [D_{\tau,d}(S)-D_{\tau,d}(S')],
     \end{align}
    where 
    $\displaystyle \mathbf{f}(\infty;\, d,\ell) 
    := \lim_{n\to\infty} n^{-\frac{|\ell|}{1+\alpha}} \cdot \mathbf{f}(k,n;\, d,\ell) 
    =  C(k,|\ell|) \cdot\prod_{u \in \tau}\ [d(u) +\ell(u)+\alpha-2]_{\ell(u)}.$
    Thus, our goal is to find a decoration $\ell : \tau \to \bbN$ with $\ell(u) \geq 2$ for all $u \in \tau$, $|\ell| > 1 + \alpha$ 
    and such that the right-hand side of \eqref{eq:f_D} is non-zero. 
    
%    Considering Theorem~\ref{thm:second_moment} and Lemma~\ref{lem:d_tv_moments}, 
%    it suffices to find a decoration $\ell$ of $\tau$ 
%    with $\ell(u) \geq 2$ for all $u\in \tau$, $|\ell| > 1 + \alpha$ and such that
%    $\sum_{d: \tau \to \bbN^*} \mathbf{f}(\infty;\, d,\ell) \cdot [D_{\tau,d}(S)-D_{\tau,d}(S')] \neq 0$.
%    Indeed, for any such decoration, we would have 
%    \begin{align*}
%    	\lim_{n \to \infty} n^{-\frac{|\ell|}{1+\alpha}}   
%    	\big( \bbE[\mathcal{F}_{\btau}(T_n^{S})]-\bbE[\mathcal{F}_{\btau}(T_n^{S'})] \big)  \neq 0
%    	\quad \text{and}\quad 
%    	 \bbE[\mathcal{F}_{\btau}(T_n^{S})^2] + \bbE[\mathcal{F}_{\btau}(T_n^{S'})^2] 
%    	 = \calO\big(n^{\frac{2|\ell|}{1+\alpha}}\big).
%    \end{align*}
%    Lemma~\ref{lem:d_tv_moments} then yields
%    $\displaystyle\limsup_{n\to \infty}d_{TV}(\mathcal{F}_{\btau}(T_n^{S}),\mathcal{F}_{\btau}(T_n^{S'})) >0$, 
%    which in turn implies Theorem~\ref{thm:main}. 
%    \smallskip 
    
    Fix some arbitrary order $u_1,\dotsc,u_{r}$ for the vertices of $\tau$ (where $r:=\lvert \tau \rvert$)
    and write $\prec_\ell$ for the lexicographical order induced on the set of decorations of $\tau$. 
    Consider the set $\Delta$ of decorations $d$ of $\tau$ with $D_{\tau,d}(S) \neq D_{\tau,d}(S')$. 
    This set is finite, since $D_{\tau,d}(S)=D_{\tau,d}(S') = 0$ for all large enough decorations. Moreover, it is not void since $\tau$ is not $(S,S')$-blind. 
    Let $d_{\max}$ be the maximal element of $\Delta$ for $\prec_\ell$. 
    
    Consider some $d \in \Delta$ with $d \neq d_{\max}$ and let $j = \min\{ i : d(u_i) \neq d_{\max}(u_i)\}$. 
    Then 
    \begin{align} \label{eq:Gamma}
    	\frac{ \mathbf{f}(\infty;\, d,\ell)}{ \mathbf{f}(\infty;\, d_{\max},\ell)}
    	&= \prod_{i \geq j} 
    	\frac{[d(u_i) +\ell(u_i)+\alpha-2]_{\ell(u_i)}}{[d_{\max}(u_i) +\ell(u_i)+\alpha-2]_{\ell(u_i)}}	
    	\xrightarrow{} 0	
    \end{align} 
    when all $\ell(u_i)$ for $i>j$ are fixed and $\ell(u_j)\to \infty$. 
    
    Write $\Delta_j$ for the set of decorations $d \in \Delta$ with $d \neq d_{\max}$ and such that $\min\{ i : d(u_i) \neq d_{\max}(u_i)\} = j$. 
    Then $\Delta\setminus\{d_{\max}\} = \bigsqcup_{j=1}^{r} \Delta_j$. 
    Now let us construct a decoration $\ell$ with the necessary requirements by first choosing $\ell(u_r)$, then $\ell(u_{r-1})$ etc. 
    
    First fix $\ell(u_r) > 1 + \alpha$ large enough so that 
    \begin{align*}
    	\Big| \sum_{d \in \Delta_r} \frac{\mathbf{f}(\infty;\, d,\ell)}{\mathbf{f}(\infty;\, d_{\max},\ell)} 
    	\cdot [D_{\tau,d}(S)-D_{\tau,d}(S')]\Big| \leq \frac{1}{2r}.
    \end{align*}
    This is possible by~\eqref{eq:Gamma}. 
    Observe that the values of $\ell(u_i)$ for $i <r$ are irrelevant, as they do not appear in the above. 
    Once $\ell(u_r)$ is fixed, fix $\ell(u_{r-1}) \geq 2$ so that 
    \begin{align*}
    	\Big| \sum_{d \in \Delta_{r-1}} \frac{\mathbf{f}(\infty;\, d,\ell)}{\mathbf{f}(\infty;\, d_{\max},\ell)} 
    	\cdot [D_{\tau,d}(S)-D_{\tau,d}(S')]\Big| \leq \frac{1}{2r}.
    \end{align*}
    Again,~\eqref{eq:Gamma} shows that the values $\{\ell(u_i); i <r-1\}$ 
    do not appear in the above and that such a choice of $\ell(u_{r-1})$ is possible regardless of the value of $\ell(u_r)$ fixed before. 
    
    Continue as such until all values $\{\ell(u):u \in \tau\}$ are fixed. 
    The resulting decoration $\ell$ has $\ell(u) \geq 2$ for all $u \in \tau$ \footnote{ $\ell(u_r) \geq 2$ is ensured by the fact that $\ell(u_r) > 1+\alpha$}, it is such that $|\ell| \geq \ell(u_r) > 1+\alpha$, and satisfies
    \begin{align*}
    	\Big| \sum_{d \in \Delta_{j}} \frac{\mathbf{f}(\infty;\, d,\ell)}{\mathbf{f}(\infty;\, d_{\max},\ell)} 
    	\cdot [D_{\tau,d}(S)-D_{\tau,d}(S')]\Big| \leq \frac{1}{2r}, \qquad \forall j =1,\dots, r.
    \end{align*}
    When summing the above, we find
    \begin{align*}
    	\Big| \sum_{d \in \Delta \setminus \{d_{\max}\}} \mathbf{f}(\infty;\, d,\ell)\cdot [D_{\tau,d}(S)-D_{\tau,d}(S')]\Big| 
    	\leq \tfrac{1}{2}\, \mathbf{f}(\infty;\, d_{\max},\ell).
    \end{align*}
    Now, since $D_{\tau,d_{\max}}(S) \neq D_{\tau,d_{\max}}(S')$, this implies that 
    \begin{align*}
    	\Big| \sum_{d :\tau \to \bbN^* } \mathbf{f}(\infty;\, d,\ell)\cdot [D_{\tau,d}(S)-D_{\tau,d}(S')]\Big| 
    	\geq \tfrac{1}{2}\, \mathbf{f}(\infty;\, d_{\max},\ell) > 0, 
    \end{align*}
    which was the desired condition. 
\end{proof}

%\subsection{Proof of Theorem~\ref{thm:difference_from_seed} for seeds of different size}

\begin{proof}[Proof of Theorem~\ref{thm:difference_from_seed} for seeds of different sizes]
    Fix two seeds $S,S'$ of sizes $k$ and $k'$, respectively, with $3 \leq k' <k$. 
    Due to the Markov property, the law of $(T_n^{S'})_{n \geq k}$ is a linear combination of laws $(T_n^{R})_{n \geq k}$ 
    with $R$ ranging over the different values taken by $T_k^{S'}$.
    Write $\calT_k$ for the set of trees of size $k$ with $\bbP(T_k^{S'} = R) \neq 0$. 
    
    First we claim that there exists two trees in $\calT_k$ with different maximal degree. 
    Indeed, one possible way of going from $S'$ to a tree in $\calT_k$ is to always attach the new vertices to leaves.
    In this scenario, the maximal degree of the resulting tree is the same as that of $S'$ (since that in $S'$ is at least $2$). 
    In conclusion 
    \begin{align*}
    	\bbP\big( \max \{ \deg_{T_k^{S'}}(u):\, u \in T_k^{S'}\} = \max \{ \deg_{S'}(u):\, u \in {S'}\}\big) > 0.
    \end{align*}
    Another is to always attach the new vertex to the one of maximal degree. The resulting maximal degree in $T_k^{S'}$ would then be $k - k'$ more than that in $S'$: 
    \begin{align*}
    	\bbP\big( \max \{ \deg_{T_k^{S'}}(u):\, u \in T_k^{S'}\} = \max \{ \deg_{S'}(u):\, u \in {S'}\} + k-k'\big) > 0.
    \end{align*} 
    Thus, there exist two trees in $\calT_k$ with distinct maximal degree, as claimed. 
    
    Fix $\tau$ to be the tree formed of a single vertex. 
    A decoration  for $\tau$ is then simply an integer number. 
    This tree has no subtree, hence Proposition~\ref{propositionblindtreePAmodel} applies to it. 
    For any decoration $\ell \in \bbN$ of~$\tau$, equation~\eqref{eq:f_D} adapts to
    \begin{align*}
     	\lim_{n \to \infty} n^{-\frac{\ell}{1+\alpha}} \,\big( \bbE[\mathcal{F}_{\btau}(T_n^{S})]-\bbE[\mathcal{F}_{\btau}(T_n^{S'})] \big) 
    	=  \sum_{d \in \bbN^*} \mathbf{f}(\infty;\, d,\ell) 
    	\sum_{R \in \calT_k} \bbP(T_k^{S'} = R)\cdot [D_{d}(S)-D_{d}(R)],
     \end{align*}
    where
    $\displaystyle \mathbf{f}(\infty;\, d,\ell) = C(k,\ell) \cdot[d +\ell+\alpha-2]_{\ell}$
    and $D_{d}(S)$ is simply the number of vertices of degree $d$ in~$S$. 
    
    Write $d_{\max}$ for the maximal number $d \in \bbN^*$ with 
    \begin{align}\label{eq:diff_deg_seq}
    	\sum_{R \in \calT_k} \bbP(T_k^{S'} = R)\cdot [D_{d}(S)-D_{d}(R)] \neq 0.
    \end{align}
    First observe that the set of values $d$ satisfying~\eqref{eq:diff_deg_seq} is finite 
    since $D_{d}(S) = D_{d}(R) =0$ for all $R \in \calT_k$ provided $d$ is large enough. 
    Second, notice that there exist at least one $d$ satisfying~\eqref{eq:diff_deg_seq} since the maximal degree of $R \in \calT_k$ is not constant. 
    
    Now, by~\eqref{eq:Gamma}, 
    \begin{align*}
    	\lim_{\ell\to\infty}
    	\sum_{d \in \bbN^*} \frac{\mathbf{f}(\infty;\, d,\ell)}{\mathbf{f}(\infty;\, d_{\max},\ell) }
    	\sum_{R \in \calT_k} \bbP(T_k^{S'} = R)\cdot [D_{d}(S)-D_{d}(R)]
    	&\\
    	=\sum_{R \in \calT_k} \bbP(T_k^{S'} = R)\cdot [D_{d_{\max}}(S)-D_{d_{\max}}(R)] &\neq 0.
    \end{align*}
    In conclusion, one may fix $\ell > 1+ \alpha$ (and implicitly $\ell \geq 2$) so that 
    \begin{align*}
    	\lim_{n \to \infty} n^{-\frac{\ell}{1+\alpha}} \,\big( \bbE[\mathcal{F}_{\btau}(T_n^{S})]-\bbE[\mathcal{F}_{\btau}(T_n^{S'})] \big)
    	% =\sum_{d \in \bbN^*} \mathbf{f}(\infty;\, d,\ell)\sum_{R \in \calT_k} \bbP(T_k^{S'} = R)\cdot [D_{d}(S)-D_{d}(R)] 
    	\neq 0. 
    \end{align*}
%    Moreover, by Theorem~\ref{thm:second_moment}, we have
%    \begin{align*}
%       \bbE\big[\calF_{\btau}(T_{n}^S)^2\big] = \calO\big( n^{\frac{2 |\ell|}{1+\alpha}}\big) \quad \text{and} \quad
%        \bbE\big[\calF_{\btau}(T_{n}^{S'})^2\big] = \calO\big( n^{\frac{2 |\ell|}{1+\alpha}}\big).
%    \end{align*}
%    Lemma~\ref{lem:d_tv_moments} then yields
%    $\displaystyle\limsup_{n\to \infty}d_{TV}(\mathcal{F}_{\btau}(T_n^{S}),\mathcal{F}_{\btau}(T_n^{S'})) >0$, 
%    which in turn implies Theorem~\ref{thm:main}. 
\end{proof}

\subsection{Proof of Theorem~\ref{thm:main}}\label{sec:proof_main}

\begin{proof}
	Fix $S,S'$ two distinct seeds of sizes at least $3$. 
	Let $\btau$ be the decorated tree given by Theorem~\ref{thm:difference_from_seed} for these two seeds. 
	Due to our assumptions on $\btau$, Theorem~\ref{thm:second_moment} applies to it, and we have 
    \begin{align*}
       \bbE\big[\calF_{\btau}(T_{n}^S)^2\big] = \calO\big( n^{\frac{2 |\ell|}{1+\alpha}}\big) \quad \text{and} \quad
        \bbE\big[\calF_{\btau}(T_{n}^{S'})^2\big] = \calO\big( n^{\frac{2 |\ell|}{1+\alpha}}\big).
    \end{align*}
    Lemma~\ref{lem:d_tv_moments} then yields
    $\displaystyle\liminf_{n\to \infty}d_{TV}(\mathcal{F}_{\btau}(T_n^{S}),\mathcal{F}_{\btau}(T_n^{S'})) >0$, 
    which in turn implies Theorem~\ref{thm:main}. 
\end{proof}

\section{Open problems and future research}\label{sec:open_problems}

\paragraph{Other attachment mechanisms}
As mentioned in the introduction, our approach uses the fact that the attachment mechanism is affine to couple the evolution of trees starting from distinct seeds of same size; 
see the coupling of Section~\ref{sec:coupling} and its use in Section~\ref{afirstgeneralresultaroundtheseed} and \ref{sec:blindtreesandarefinement} to deduce Corollary~\ref{cor:difference_in_seed}.

Suppose now that we consider a different attachment model, 
where the new vertex is attached to a vertex $u \in T_n$ with probability proportional to $g(\deg_{T_n}(u))$ for some function $g :\bbN \to (0,+\infty)$. 
Thus \eqref{probabilitylawvertexaffinePAmodel} becomes
\begin{align*}
	\mathbb{P}(u_N = u \,\vert\,T_{k}^{S}, \dots, T_{N}^{S} \text{ with } T_{N}^{S}=T) 
	= \frac{g(\deg_{T}(u))}{\sum_{v\in T}g(\deg_{T}(v))} \qquad \forall u\in T.
\end{align*}
Notice that, when $g$ is not affine, the denominator above depends on the structure of $T$, 
and the attachment probabilities cease to be a local function of $u$. In other words, the sequence of drawn vertices loses its exchangeability. 
As a consequence, it is not possible anymore to construct a coupling between sequences 
$(T_{n}^{S})_{n \geq k}$ and $(T_{n}^{S'})_{n \geq k}$ as in Section \ref{sec:coupling}, where $S$ and $S'$ are two seeds of size~$k$. 
Nevertheless, if $S$ and $S'$ have same degree sequences, then regardless of the form of $g$, 
one may couple the evolution of $T_n^S$ and $T_n^{S'}$ so that the sets of connected components of $T_n^S \setminus S$ and $T_n^{S'} \setminus S'$ are identical 
(here $T_n^S\setminus S$ stands for the graph obtained from $T_n^S$ by removing all edges of $S$).
The equivalent of Corollary~\ref{cor:difference_in_seed} may then be proved in the same manner. 

This hints to the possibility of seed recognition (\emph{i.e.} Theorem \ref{thm:main}) for a much larger set of models. 
One particular example of interest is when $g(k)= k^{\beta}$ for some $\beta \in (0,1)$ \cite{krapivsky2000connectivity}. 
In such models, the largest degree in $T_n^S$ is of order $(\log n)^{\frac1{1-\beta}}$, hence much smaller than in the affine case \cite[Thm.~22]{Bha07}. 
The case $\beta = 1$ is that of the linear preferential attachment model treated in \cite{curien2015scaling}; 
when $\beta = 0$ we obtain the uniform attachment model of \cite{bubeck2017}.
For $\beta \in (0,1)$ we expect the same type of result to hold, and plan to investigate this in future work. 

Let us also mention that, when $\beta>1$ a single vertex of $T_n^{S}$ has degree tending to infinity, all other degrees are a.s. bounded by a constant, as is the diameter of $T_n^{S}$ \cite[Thm.~1.2]{Oli05}. 

\paragraph{Finding the seed}
Our result may be understood as follows: given a large (but uniform in $n$) number of samples of $T_n^{S}$, one may recover $S$ with high precision.
A related question is to locate $S$ given a single instance of $T_n^{S}$. 
One may not hope to do this with high probability, but is it possible to do it with uniformly positive probability? 
Results in this direction were obtained in \cite{Bub17,dev18,Lug18} for the uniform and linear preferential attachment models. 

One may ask whether having a uniformly positive chance of locating the seed is equivalent to Theorem~\ref{thm:main}.
Moreover, for specific cases such as the $\alpha$-PA, is there an explicit algorithm that locates the seed?

\paragraph{Beyond trees}
Finally, one may consider randomly growing graphs, rather than trees. 
Indeed, imagine a model where vertices are added one by one to a growing graph, with each new vertex being attached to each old vertex independently, 
with a probability depending on the size of the graph and on the degree of the old vertex. 
This offers great freedom for the choice of the attachment probability, but the resulting graph ceases to be a tree; it may even have multiple connected components. These aspects render the study of such models more delicate. However, we believe that for certain attachment mechanisms - affine for instance - an equivalent of Theorem~\ref{thm:main} would remain valid.

\bibliographystyle{plain}
\bibliography{biblioaffinecase}

\begin{thebibliography}{10}

\bibitem{barabasi1999emergence}
Albert-L{\'a}szl{\'o} Barab{\'a}si and R{\'e}ka Albert.
\newblock Emergence of scaling in random networks.
\newblock {\em Science}, 286(5439):509--512, 1999.

\bibitem{berger2005spread}
Noam Berger, Christian Borgs, Jennifer~T Chayes, and Amin Saberi.
\newblock On the spread of viruses on the internet.
\newblock In {\em Proceedings of the sixteenth annual ACM-SIAM symposium on
  Discrete algorithms}, pages 301--310. Society for Industrial and Applied
  Mathematics, 2005.

\bibitem{Bha07}
Shankar Bhamidi.
\newblock Universal techniques to analyze preferential attachment trees:
  {G}lobal and {L}ocal analysis.
\newblock 2007.
\newblock preprint.

\bibitem{bollobas2001degree}
B{\'e}la Bollob{\'a}s, Oliver Riordan, Joel Spencer, G{\'a}bor Tusn{\'a}dy,
  et~al.
\newblock The degree sequence of a scale-free random graph process.
\newblock {\em Random Structures \& Algorithms}, 18(3):279--290, 2001.

\bibitem{Bub17}
S{\'e}bastien Bubeck, Luc Devroye, and G{\'a}bor Lugosi.
\newblock Finding {A}dam in random growing trees.
\newblock {\em Random Structures \& Algorithms}, 50(2):158--172, 2017.

\bibitem{bubeck2017}
S{\'e}bastien Bubeck, Ronen Eldan, Elchanan Mossel, and Mikl{\'o}s~Z. R{\'a}cz.
\newblock From trees to seeds: On the inference of the seed from large trees in
  the uniform attachment model.
\newblock {\em Bernoulli}, 23(4A):2887--2916, 11 2017.

\bibitem{bubeck2015influence}
S{\'e}bastien Bubeck, Elchanan Mossel, and Mikl{\'o}s~Z R{\'a}cz.
\newblock On the influence of the seed graph in the preferential attachment
  model.
\newblock {\em IEEE Transactions on Network Science and Engineering},
  2(1):30--39, 2015.

\bibitem{curien2015scaling}
Nicolas Curien, Thomas Duquesne, Igor Kortchemski, and Ioan Manolescu.
\newblock Scaling limits and influence of the seed graph in preferential
  attachment trees (limites d'{\'e}chelle et ontogen{\`e}se des arbres
  construits par attachement pr{\'e}f{\'e}rentiel).
\newblock {\em Journal de l'{\'E}cole polytechnique---Math{\'e}matiques},
  2:1--34, 2015.

\bibitem{dev18}
Luc Devroye and Tommy Reddad.
\newblock On the discovery of the seed in uniform attachment trees.
\newblock {\em arXiv preprint arXiv:1810.00969}, 2018.

\bibitem{hofstad_2016}
Remco van~der Hofstad.
\newblock {\em Random Graphs and Complex Networks}, volume~1 of {\em Cambridge
  Series in Statistical and Probabilistic Mathematics}.
\newblock Cambridge University Press, 2016.

\bibitem{krapivsky2000connectivity}
Paul~L Krapivsky, Sidney Redner, and Francois Leyvraz.
\newblock Connectivity of growing random networks.
\newblock {\em Physical review letters}, 85(21):4629, 2000.

\bibitem{Lug18}
Gabor Lugosi and Alan~S Pereira.
\newblock Finding the seed of uniform attachment trees.
\newblock {\em arXiv preprint arXiv:1801.01816}, 2018.

\bibitem{middendorf2005inferring}
Manuel Middendorf, Etay Ziv, and Chris~H Wiggins.
\newblock Inferring network mechanisms: the drosophila melanogaster protein
  interaction network.
\newblock {\em Proceedings of the National Academy of Sciences of the United
  States of America}, 102(9):3192--3197, 2005.

\bibitem{mori2002random}
T~M{\'o}ri.
\newblock On random trees.
\newblock {\em Studia Scientiarum Mathematicarum Hungarica}, 39(1-2):143--155,
  2002.

\bibitem{newman2018networks}
Mark Newman.
\newblock {\em Networks}.
\newblock Oxford university press, 2018.

\bibitem{Oli05}
Roberto Oliveira and Joel Spencer.
\newblock Connectivity transitions in networks with super-linear preferential
  attachment.
\newblock {\em Internet Math.}, 2(2):121--163, 2005.

\bibitem{Wang08}
Mingyang Wang, Guang Yu, and Daren Yu.
\newblock Measuring the preferential attachment mechanism in citation networks.
\newblock {\em Physica A: Statistical Mechanics and its Applications},
  387(18):4692 -- 4698, 2008.

\end{thebibliography}

\end{document}